\documentclass{anmathcs}
\usepackage{mathrsfs, amssymb, eucal, amsmath, amsthm, amscd, hyperref}
\usepackage{graphicx}
\usepackage{bm}
\usepackage{cite}
\usepackage{color}
\usepackage{hyperref}
\hypersetup{
    colorlinks=true,
    linkcolor=blue,
    filecolor=magenta,
    urlcolor=blue,
    citecolor=red,
    }
\usepackage[dvipsnames]{xcolor}
\usepackage{amssymb,amsmath,amsthm,latexsym}
\usepackage{amsmath}
\colorlet{LightRubineRed}{RubineRed!70}
\colorlet{Mycolor1}{green!10!orange}
\usepackage{lipsum}
\newcommand\blfootnote[1]{%
  \begingroup
  \renewcommand\thefootnote{}\footnote{#1}%
  \addtocounter{footnote}{-1}%
  \endgroup
}
\usepackage{booktabs}
\usepackage{array}
\usepackage{longtable}

\newtheorem{theorem}{Theorem}

\newtheorem{corollary}[theorem] {Corollary}
\newtheorem{lemma} [theorem]{Lemma}

\newtheorem{remark}[theorem]{Remark}
\numberwithin{equation}{section}
\newtheorem{definition}[theorem]{Definition}

\usepackage{geometry}
\numberwithin{equation}{section}

\theoremstyle{remark}

\begin{document}
\title[On Characterizing Potential Friends of 20]{On Characterizing Potential Friends of 20}
\author[Tapas Chatterjee]{Tapas Chatterjee}
\address{Department of Mathematics, Indian Institute of Technology Ropar\\ Punjab, India}
\email{tapasc@iitrpr.ac.in}
\author[Sagar Mandal]{Sagar Mandal}
\address{Department of Mathematics and Statistics, Indian Institute of Technology Kanpur\\ Kalyanpur, Kanpur, Uttar Pradesh 208016, India}
\email{sagarmandal31415@gmail.com}
\author[Sourav Mandal]{Sourav Mandal}
\address{Department of Mathematics, Ramakrishna Mission Vivekananda Educational and Research Institute, Howrah}
\email{souravmandal1729@gmail.com}


\maketitle


\begin{abstract}
Does $20$ have a friend? Or is it a solitary number? A folklore conjecture asserts that $20$ has no friends i.e. it is a solitary number. In this article, we prove that, a friend $N$ of $20$ is of the form $N=2\cdot5^{2a}\cdot m^2$, with $(3,m)=(7,m)=1$ and it has at least six distinct prime divisors. Furthermore, we show that  $\Omega(N)\geq 2\omega(N)+6a-5$ and if $\Omega(m)\leq K$ then $N< 10\cdot 6^{(2^{K-2a+3}-1)^2}$, where $\Omega(n)$ and $\omega(n)$ denote the total number of prime divisors and the number of distinct prime divisors of the integer $n$ respectively. In addition, we deduce that, not all exponents of odd prime divisors of friend $N$ of $20$ are congruent to $-1$ modulo $f$, where $f$ is the order of $5$ in $(\mathbb{Z}/p\mathbb{Z})^\times$ such that $3\mid f$ and  $p$ is a prime congruent to $1$ modulo $6$. Also, we prove necessary upper bounds for all prime divisors of friends of 20 in terms of the number of divisors of the friend. In addition, we prove that, if $P$ is the largest prime divisor of $N$ then $P<N^{\frac{1}{4}}$.
\end{abstract}

\blfootnote{\hspace{-8.2mm} \copyright\, 2025\, The Authors.}

\section{Introduction}

Two distinct positive integers $a$ and $b$ are said to be friendly if $I(a)=I(b)$, where $I(q)$ is the abundancy index of $q$, which is defined as $I(q)=\sigma(q)/{q}$, where $\sigma(q)$ is the sum of positive divisors of $q$. Solitary number refers to a number that has no friends. A number $n$ is said to be perfect if it has an abundancy index of $2$. A number $n$ is said to be deficient or defective if its abundancy index is less than $2$. A number $n$ is said to be abundant or excessive if its abundance index exceeds $2$.\\
The number $10$ is the smallest known suspected solitary number. Various properties and constraints on a potential friend of $10$ have been explored in \cite{6,ward,8,9,10,11}, but no such number has been found. Moreover, $10$ is not the only suspected solitary number, several others (e.g. $14,15,20$ etc.) are known, though their status remains unresolved. The only large known family of solitary numbers comes from $n$ such that $\sigma(n)$ and $n$ are relatively prime.\\

A positive integer $N>20$ is said to be a friend of $20$ if $I(N)=I(20)=21/10$. It is believed that whether 20 is solitary or not is as difficult as to finding odd perfect numbers. If any of the suspected solitary numbers up to $372$ is actually a friendly number, then its smallest friend must be strictly greater than $10^{30}$\cite{OEIS}. In this article, we prove the following results related to friends of 20
\begin{theorem}\label{thm 1.1}
A friend $N$ of $20$ is of the form $N=2\cdot 5^{2a}m^2$,  with $(3,m)=(7,m)=1$, where $a,m \in \mathbb{Z}^{+}$. Further, $N$ has at least six distinct prime divisors.
\end{theorem}
The above theorem depicts that 20 cannot have an odd friend. In the next theorem, we give a lower bound for the total number of prime divisors of any such friend $N$ of 20. Here is the statement of the theorem: 
 \begin{theorem}\label{thm 1.2}
    Let $N=2\cdot5^{2a}m^2$ be a friend of 20, then $$\Omega(N)\geq 2\omega(N)+6a-5$$ where $\Omega(n)$ and $\omega(n)$ denote the total number of prime divisors and the number of distinct prime divisors of the integer $n$ respectively.
  \end{theorem}
   As an immediate corollary, we get the following upper bound for $N$. 
  
\begin{corollary}\label{coro 1.3}
     If $N=2\cdot5^{2a}m^2$ is a friend of 20 where $a,m \in \mathbb{Z}^{+}$ and $\Omega(m)\leq K$ then $$N< 10\cdot 6^{(2^{K-2a+3}-1)^2}.$$
    
\end{corollary}
In addition, we also give a lower bound for $N$ using the following theorem:   
\begin{theorem}\label{thm 1.4}
Let $p\equiv1 \pmod 6$ be a prime, also let order of $5$ in $(\mathbb{Z}/p\mathbb{Z})^\times$ be $f$ such that $3\mid f$. If $v_5(N)\equiv -1 \pmod f$, where $N$ is a friend of 20, then there exists an odd prime divisor of $N$ (say $q$) such that $v_{q}(N) \not \equiv -1 \pmod f$, where $v_p(N)=k$ denotes for $p^k \mid N$ but $p^{k+1} \nmid N$.  
\end{theorem}
\begin{remark}
The preceding theorem implies that not all exponents of odd prime divisors of friend $N$ of $20$ are congruent to $-1$ modulo $f$, where
the order of $5$ in $(\mathbb{Z}/p\mathbb{Z})^\times$ is $f$ such that $3\mid f$ and  $p\equiv1 \pmod 6$ is a prime.
\end{remark}
\begin{corollary}\label{coro 1.6}
If $N=2F^2$ is a friend of $20$, then $F$ is not square-free. 
\end{corollary}

\begin{theorem}\label{thm 1.8}
   Let $N$ be a friend of 20 and $q_r$ ($r\geq 3$) be the $r$-th smallest prime divisor of $N$. Then necessarily 
$q_r< L (\log L+\log\log L),$
where 
$$L=\lceil\frac{\mathcal{U}\omega(N)}{\mathcal{V}}\rceil,~~~~ \frac{\mathcal{U}}{\mathcal{V}} > \frac{1}{\frac{28}{25}\cdot \displaystyle{\prod_{5\leq i\leq r+1} (1-\frac{1}{p_i})-1}}$$ 
where $p_i$ is the $i$-th prime number, $\frac{\mathcal{U}}{\mathcal{V}}\in \mathbb{Q}^{+}\setminus \mathbb{Z}^{+}$, and $(\mathcal{U},\mathcal{V})=1$ such that $\mathcal{UV}(r-2)+2\mathcal{U}+\mathcal{V}>\mathcal{V}^2$.
\end{theorem}
\begin{theorem}\label{newthm}
Let $N$ be a friend of $20$ and let $P$ be the largest prime divisor of $N$. Then $P<N^{\frac{1}{4}}$.    
\end{theorem}
\section{Definition and Notation}
In this section, we introduce some notations and definitions.\\
Let $q$ be an odd prime. Define
    $$\mathcal{F}_{q}(x)=\{p : \text{$p$ odd prime}, p \mid x, q\mid \sigma(p^\eta), \eta \geq 2, \eta~ \text{even}\}$$
    and  $|\mathcal{F}_q(N)|$ denotes the cardinality of $\mathcal{F}_q(N)$.\\\\
    Let  $$\mathcal{A}_{n, a}(r):=\left \{\sum_{i=1}^{r}a^{c_i}-r :\sum_{i=1}^{r} c_i =n, c_i \in \mathbb{N} \right \}$$~~\text{and}~~ $$\mathcal{H}_{n,a}:=\bigcup_{r=1}^{n}\mathcal{A}_{n,a}(r).$$
   
   \textbf{Notation:}
    \begin{itemize}
   \item $v_p(N)$ denotes the integer $k$ such that $p^{k}\mid N$ but $p^{k+1} \nmid N$.
        \item $f_p^q$ is denoted for the multiplicative order of $q$ modulo $p$.
    \end{itemize}
Throughout this article, we use $p,p_1,\dots,p_{\omega(N)},q,q_1,\dots,q_{\omega(N)}$ for denoting prime numbers. Further, we assume that the numbers $a,b,m,a_1,\dots,a_{\omega(N)}$ are positive integers.

\section{Preliminaries}
In this section, we state some of the useful results which will play a significant role in proving our main theorems.

\begin{lemma}[\cite{2,7}]\label{lem 3.1}~~~~~~~~
\begin{enumerate}

    \item $I(n)$ is weakly multiplicative, that is, if $n$ and $m$ are two coprime positive integers then $I(nm)=I(n)I(m)$.
    \item Let $a,n$ be two positive integers and $a>1$. Then $I(an)>I(n)$.
    \item Let $p_1$, $p_2$, $p_3$,\dots, $p_m$ be $m$ distinct primes and $a_1$, $a_2$, $a_3$ ,\dots, $a_m$ be positive integers then
\begin{align*}
    I\biggl (\prod_{i=1}^{m}p_i^{a_i}\biggl)=\prod_{i=1}^{m}\biggl(\sum_{j=0}^{a_i}p_i^{-j}\biggl)=\prod_{i=1}^{m}\frac{p_i^{a_i+1}-1}{p_i^{a_i}(p_i-1)}.
\end{align*}
\item \label{p4}If $p_{1}$,...,$p_{m}$ are distinct primes, $q_{1}$,...,$q_{m}$ are distinct primes such that  $p_{i}\leq q_{i}$ for all $1\leq i\leq m$. If $t_1,t_2,...,t_m$ are positive integers then
\begin{align*}
   I \biggl(\prod_{i=1}^{m}p_i^{t_i}\biggl)\geq I\biggl(\prod_{i=1}^{m}q_i^{t_i}\biggl).
\end{align*}
\item \label{p5} If $n=\displaystyle{\prod_{i=1}^{m}p_i^{a_i}}$, then $I(n)<\displaystyle{\prod_{i=1}^{m}\frac{p_i}{p_i-1}}$.

\end{enumerate}
\end{lemma}
\begin{lemma}[\cite{11}, Theorem 1.1]\label{ss2024thm1.1}
 Let $p,q$ be two distinct prime numbers with  $p^{k-1}\mid\mid (q-1)$ where $k$ is some positive integer. Then $p$ divides $\sigma(q^{2a})$ if and only if $2a+1 \equiv\ 0 \pmod f$ where $f$ is the smallest odd positive integer greater than $1$ such that $q^{f}\equiv 1 \pmod {p^{k}}$. 

\end{lemma}

\begin{corollary}[\cite{11}, Corollary 1.2]\label{co3.5}
 Let $p$, $p^*$ and $q$ be three distinct prime numbers with $p^{k-1}\mid \mid (q-1)$ and $p^{*k^*-1}\mid \mid (q-1)$ $(k,k^*\in \mathbb{Z^{+}}$) also let $p \mid \sigma(q^{2a})$ and $p^* \mid \sigma(q^{2a^*})$( a,$a^*$$\in \mathbb{Z^{+}}$) with $f_{p}^{q}$, $f_{p^*}^{q}$ respectively in Lemma \ref{ss2024thm1.1}. If $f_{p^*}^{q}\mid f_{p}^{q}$ then $pp^*\mid \sigma(q^{2a})$.  
\end{corollary}

\begin{lemma}[\cite{11}, Theorem 1.5]\label{ss}

If $p$ and $q$ are two distinct prime numbers with $q>p$, then $p$ divides $\sigma(q^{2a})$ if and only if for $r\neq 1$, $q\equiv r \pmod p$ and $2a+1 \equiv 0 \pmod f$ where $f$ is the smallest odd positive integer  greater than $1$ such that $r^f\equiv 1 \pmod p$ and for $r=1$, $q\equiv r \pmod p$ and $2a+1\equiv 0 \pmod p$.

\end{lemma}

\begin{corollary}\label{cor12}
 $\sigma(q^{2a})$ is divisible by $5$ if and only if $q \equiv 1 \pmod{10}$ with $2a+1 \equiv 0 \pmod{5}$.   
\end{corollary}
\begin{proof}
 In Lemma \ref{ss}, we set $p=5$ and $q > 5$ then as for $r= 2, 3, 4$ we have no odd integer $f >1$ such that $r^f \equiv 1 \pmod 5$, hence the only choice for $r$ is $1$, and therefore the result follows.
\end{proof}

\begin{lemma}\label{lemma13}
    If $p_n$ is $n\text{-th}$ prime number then
    $$p_n<n(\log n+\log\log n)$$
    for $n\geq 6$.
\end{lemma}
 \begin{proof}
  See \cite{pb} for a proof.   
 \end{proof}

\begin{definition}
An odd number $M$ is said to be an odd $m/d$-perfect number if $
\frac{\sigma(M)}{M}=\frac{m}{d}.$ 
\end{definition}
  \begin{lemma}\label{pn}
  If M is an odd m/d-perfect number with k distinct prime factors then $$M < d(d+1)^{(2^k-1)^2}.$$
  \end{lemma}
  \begin{proof}
  See \cite{5} for a proof.
  \end{proof}

\section{Basic properties of friends of 20}
Before we start proving our main theorems, it is convenient to deduce some of the basic properties of friends of $20$. Note that, If $N$ is a friend of 20, then
$$I(N)=\frac{\sigma(N)}{N}=\frac{21}{10}=I(20)$$
which implies that, $$10\cdot \sigma(N)=21\cdot N$$ so $10 \mid \text{N}$. Therefore, $N$ has the form: $\text{N}=2^{a}\cdot 5^{b}\cdot m$, where $(2,m)=(5,m)=1$.

  \begin{lemma}\label{pr3.11}
Let $N$ be a friend of $20$, then $2 \mid\mid  N $.
\end{lemma}
\begin{proof}
Let $N$ be a friend of 20, then we can write $N=2^{a}\cdot 5^{b}\cdot m$ with $(2,m)=(5,m)=1$. Note that, $N>20$. If possible, assume that $a\geq 2$. Then
we have from Lemma \ref{lem 3.1} that
$I(N)>I(2^2\cdot5)=\frac{21}{10}$
but this is absurd. Therefore, $N$ cannot have more than one $2$ in its prime factorization. This completes the proof.
\end{proof}
\begin{lemma}\label{lemma17}
A friend $N$ of $20$ is of the form $N=2\cdot 5^{2a}m^2$.
\end{lemma}
\begin{proof}
 Let $N$ be a friend of $20$, with $\omega(N)=s$, then 
 $I(N)=\frac{21}{10}$
 implies,
 \begin{align}\label{1}
      10\cdot\sigma(N)=21\cdot N.
 \end{align}
Therefore using  Lemma \ref{pr3.11} and (\ref{1}) we get that $N=2\cdot 5^{a}\cdot m$. Note that $m>1$ as if $m=1$ then
$$I(N)=I(2\cdot 5^{a})=I(2)\cdot I(5^{a})<\frac{3\cdot 5}{2\cdot 4}=\frac{15}{8}<\frac{21}{10}.$$
Let $m=\displaystyle{\prod_{i=1}^{s-2}p_{i}^{a_i}}.$
Then, from (\ref{1}) we obtain
$10\cdot\sigma(2\cdot 5^{a}\cdot\prod_{i=1}^{s-2}p_{i}^{a_i})=21\cdot 2\cdot 5^{a}\cdot\displaystyle{\prod_{i=1}^{s-2}p_{i}^{a_i}},$
which implies,
$\sigma(2)\cdot\sigma(5^{a})\cdot\prod_{i=1}^{s-2}\sigma(p_{i}^{a_i})=21\cdot 5^{a-1}\cdot\prod_{i=1}^{s-2}p_{i}^{a_i}. $
Since $\sigma(2)=3$, we have
\begin{align}\label{2}
 \sigma(5^{a})\cdot\prod_{i=1}^{s-2}\sigma(p_{i}^{a_i})=7\cdot 5^{a-1}\cdot\prod_{i=1}^{s-2}p_{i}^{a_i}.   
\end{align}

If possible, suppose that one of $a,a_1,\dots,a_{s-2}$ is odd. Without loss of generality, suppose that $a_k,1\leq k\leq s-2$ is odd. Then from (\ref{2}) we can write
\begin{align}\label{3}
(1+5+...+5^{a})\cdots(1+p_k+...+p_k^{a_k})\cdots(1+p_{s-2}+...+p_{s-2}^{a_{s-2}})\\=7\cdot 5^{a-1}\cdot\prod_{i=1}^{s-2}p_{i}^{a_i}.
\end{align}
Since all $p_i>2$ and $a_k$ is odd, it implies $(1+p_k+...+p_k^{a_k})$ is even and so does the left hand side of (\ref{3}) but the right hand side of (\ref{3}) is an odd integer, which is absurd. Therefore  $a,a_1,\dots,a_{s-2}$ are even positive integers. This proves that, the form of $N$ is $N=2\cdot 5^{2a}m^2$. 

\end{proof}
\begin{remark}
If $N$ is a friend of $20$, then it is easy to see from (\ref{2}) that, the prime divisors of $\sigma(q^{2a_q})$ belong to the set $\{7,p: p\mid (\frac{N}{2})\}$, where $q^{2a_q}\mid\mid N$. 
\end{remark}

\begin{lemma}\label{lemma19}
    Let $N$ be a friend of $20$ then it is co-prime to $21$. 
\end{lemma}
\begin{proof}
At first, we show that $N$ is co-prime to $3$. Assume that $N=2\cdot 5^{2a}\cdot3^{2b}\cdot m^2$, then
    $\frac{\sigma(N)}{N}=\frac{21}{10}$
    implies
    $$\sigma(2)\cdot \sigma(5^{2a})\cdot \sigma(3^{2b})\cdot \sigma(m^2)=21\cdot 5^{2a-1}\cdot3^{2b}\cdot m^2$$
which is equivalent to,
    $$\frac{\sigma(m^2)}{m^2}=\frac{21\cdot 5^{2a-1}\cdot3^{2b}}{\sigma(2)\cdot \sigma(5^{2a})\cdot \sigma(3^{2b})}.$$
Since $\sigma(m^2)\geq m^2$, we must have
$21\cdot 5^{2a-1}\cdot3^{2b}\geq \sigma(2)\cdot \sigma(5^{2a})\cdot \sigma(3^{2b})$,
since $\sigma(2)=3$, we have
$7\cdot 5^{2a-1}\cdot3^{2b}\geq \sigma(5^{2a})\cdot \sigma(3^{2b})$
that is,
\begin{align}\label{4}
    \frac{7}{5}\geq I(5^{2a}\cdot3^{2b}).
\end{align}
By Lemma \ref{lem 3.1}, we have $\frac{7}{5}<\frac{9}{5}=I(5\cdot 3)\leq I(5^{2a}\cdot 3^{2b})$ which contradicts (\ref{4}). This proves that $(N,3)=1$.\\
Assume that $N=2\cdot 5^{2a}\cdot7^{2b}\cdot m^2$, then
    $\frac{\sigma(N)}{N}=\frac{21}{10}$
    implies
    $$\sigma(2)\cdot \sigma(5^{2a})\cdot \sigma(7^{2b})\cdot \sigma(m^2)=21\cdot 5^{2a-1}\cdot7^{2b}\cdot m^2$$
which is equivalent to,
    $$\frac{\sigma(m^2)}{m^2}=\frac{21\cdot 5^{2a-1}\cdot7^{2b}}{\sigma(2)\cdot \sigma(5^{2a})\cdot \sigma(7^{2b})}.$$
Since $\sigma(m^2)\geq m^2$, we must have
$21\cdot 5^{2a-1}\cdot7^{2b}\geq \sigma(2)\cdot \sigma(5^{2a})\cdot \sigma(7^{2b})$, since $\sigma(2)=3$, we have
$7\cdot 5^{2a-1}\cdot7^{2b}\geq \sigma(5^{2a})\cdot \sigma(7^{2b})$
that is,
\begin{align}\label{5}
    \frac{7}{5}\geq I(5^{2a}\cdot7^{2b}).
\end{align}
By Lemma \ref{lem 3.1}, we have $\frac{7}{5}<\frac{1767}{1225}=I(5^2\cdot 7^2)\leq I(5^{2a}\cdot 7^{2b})$, which contradicts (\ref{5}). This proves that $(N,7)=1$.\\
Consequently, we have $N$ is co-prime to $21$. This completes the proof.
\end{proof}

 \section{Proof of the main theorems}
\subsection{Proof of theorem \ref{thm 1.1} } 
If $N=2\cdot5^{2a}\cdot\displaystyle{\prod_{i=1}^{\omega(N)}p_i^{2a_i}}$ is a friend of $20$, then an upper bound for each prime $p_i$ can be obtained by the method discussed in the Appendix section. Note that, if any prime $p_i$ exceeds the given  bound then immediately, $I(N)<I(20)$. For more details we refer the reader to the Appendix section.\\
We shall direct mention the upper bound for each prime $p_i$ in upcoming proofs, depending on the distinct prime divisors of $N$. Further, we shall heavily use Corollary \ref{co3.5} in the proof of this theorem. The required $f_{p}^{q}$, for primes $p$ and $q$ can be found in the Appendix section. \\
In order to prove this theorem, we require the following lemmas:
\begin{lemma}\label{nlemma}
    Let $N$ be a friend of 20 then it has at least six distinct prime divisors that is $\omega(N)\geq 6$.
\end{lemma}
\begin{proof}
Let $N$ be a friend of $20$. Assume that $\omega(N)=3$, then $N=2\cdot 5^{2a}\cdot p^{2b}$ for some prime $p$. Then
$$I(N)\leq I(2\cdot5^{2a}\cdot 11^{2b})=\frac{3}{2}\cdot I(5^{2a}\cdot 11^{2b})<\frac{3\cdot 5\cdot 11}{2\cdot4\cdot10}=\frac{33}{16}<\frac{21}{10}.$$
Therefore $\omega(N)=3$ is not possible. Assume that $\omega(N)=4$, then $N=2\cdot 5^{2a}\cdot \displaystyle{\prod_{i=1}^{2}p_i^{2a_i}}$, where $p_1$ and $p_2$ are prime numbers with $p_1 < p_2$.\\
Note that $p_1\leq 17$ and $p_2 \leq 53$ (for any other choice of $p_1$ or $p_2$ gives $I(N)<I(20)$). Then
$\frac{\sigma(N)}{N}=\frac{21}{10}$
implies,
\begin{align}\label{eqnew}
\sigma(5^{2a})\cdot\sigma(p_1^{2a_1})\cdot\sigma(p_2^{2a_2})=7\cdot 5^{2a-1}\cdot p_1^{2a_1}\cdot p_2^{2a_2}.
\end{align}
Since $5 \nmid \sigma(5^{2a})$, either $5 \mid \sigma(p_1^{2a_1})$ or, $5 \mid \sigma(p_2^{2a_2})$. Note that, Corollary \ref{cor12} implies $p\equiv 1\pmod{10}$, whenever $5\mid \sigma(p^{2\alpha})$, where $p>5$ is a prime number and $\alpha\in\mathbb{Z}^{+}$. Suppose that $5\mid \sigma(p_1^{2a_1})$ then $p_1=11$ but  $71 \mid \sigma(5^{2a})$ whenever $11=p_1\mid\sigma(5^{2a})$. Thus $p_2=71$ but which is absurd as $p_2 \leq 53$. \\
This implies that $p_1$ can not be $11$ and $5\nmid\sigma(p_1^{2a_1})$. Therefore, $5\mid\sigma(p_2^{2a_2})$ (as $5\mid\sigma(5^{2a})\cdot\sigma(p_1^{2a_1})\cdot\sigma(p_2^{2a_2})$). Since $p_2\leq 53$, $p_2$ can be either $31$ or $41$. Note that, for any choice of $11\neq p_1\leq 17$, $p_1\nmid \sigma(5^{2a})$. Therefore $p_2$ must divide $\sigma(5^{2a})$, which implies that $p_2=31$. Then either $p_1=13~\text{or}~p_1=17$ but neither of these choices can divide $\sigma(31^{2a_2})$. This proves that $N$ cannot have exactly four distinct prime divisors.\\
    If possible, assume that $\omega(N)=5$, then $N=2\cdot5^{2a_1}\cdot \displaystyle{\prod_{i=1}^{3}p_i^{2\alpha_i}}$, where $p_1<p_2<p_3$. Note that $11 \leq p_1\leq 19$. We shall prove the theorem by eliminating the following possible cases: \\
\begin{center}
\textbf{\textcolor{RedViolet}{Case 1.}}\\ \textbf{\textcolor{blue}{For $p_3=11$}}\\
\end{center}

Note that, if $p_3=11$ then $13 \leq p_4\leq 109$. For $13 \leq p_4\leq 37 $, $I(N)> I(5^211^2p_2^2)>\frac{7}{5}$, therefore we shall consider $ 41 \leq p_4 \leq 109$. The following figures describe the possible subcases when $p_3=11$.\\

\begin{figure}[h]
    \centering
    \includegraphics[width=1\linewidth]{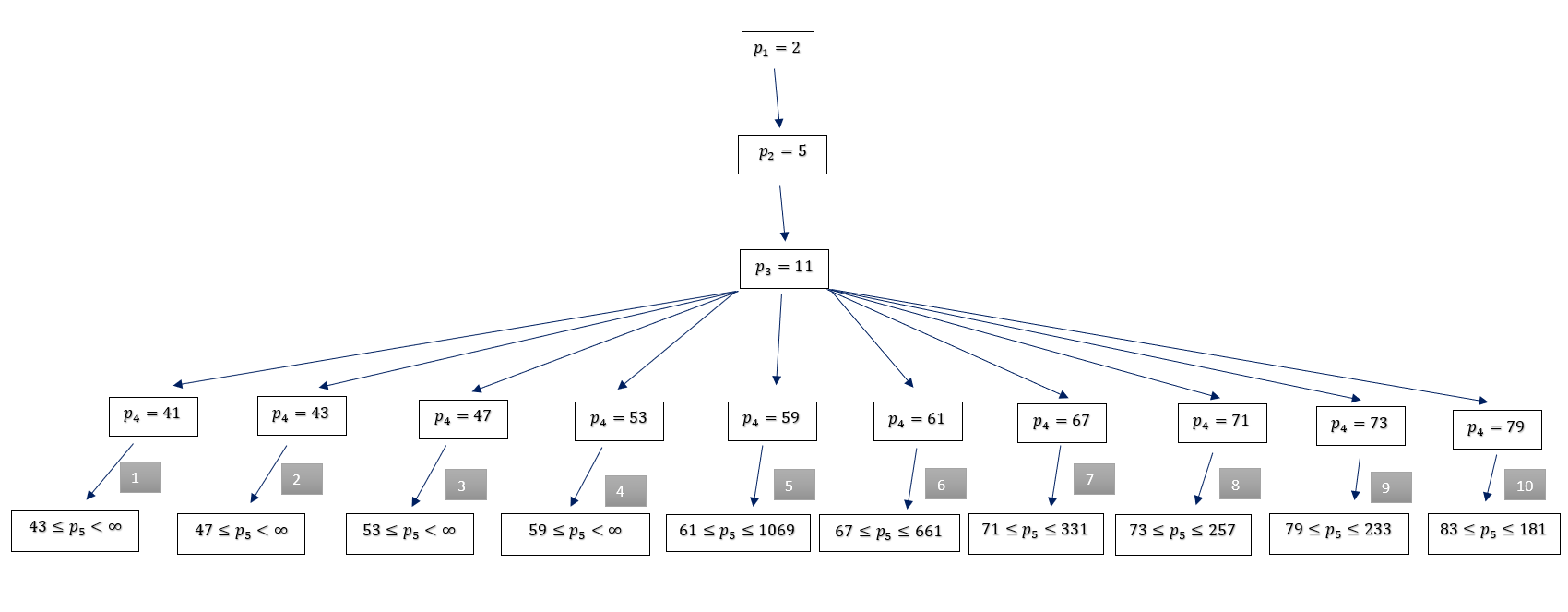}
    \caption{\textcolor{Mycolor1}{Possible subcases when $p_3=11$}}
    \label{fig:enter-label}
\end{figure}

\textbf{Subcase 1.1:}
If $N=2\cdot 5^{2a_2}\cdot11^{2a_3}\cdot41^{2a_4}\cdot p_5^{2a_5}$, then $p_5\geq 43$. Our claim is that $a_2=1$. Suppose that $a_2 \geq 2$, then
\begin{align*}
    I(N)>I(2\cdot5^4\cdot11^2\cdot41^2)&=\frac{3}{2}\cdot I(5^4\cdot11^2\cdot41^2)\\&=\frac{48810867}{23113750}>\frac{21}{10}
\end{align*}

therefore claim follows. For $a_2=1$, since $\sigma(5^{2})=31$, it implies that $p_5=31$ but this is impossible as $p_5\geq 43$. Hence, this case is impossible.\\

\textbf{Subcase 1.2:}
If $N=2\cdot 5^{2a_2}\cdot11^{2a_3}\cdot43^{2a_4}\cdot p_5^{2a_5}$, then $p_5\geq 47$. Similar argument given in subcase 1.1 implies that $N$ cannot be a friend of $20$.\\

\textbf{Subcase 1.3:}
If $N=2\cdot 5^{2a_2}\cdot11^{2a_3}\cdot47^{2a_4}\cdot p_5^{2a_5}$, then $p_5\geq 53$. Similar argument given in subcase 1.1 implies that $N$ cannot be a friend of $20$.\\

\textbf{Subcase 1.4:}
If $N=2\cdot 5^{2a_2}\cdot11^{2a_3}\cdot53^{2a_4}\cdot p_5^{2a_5}$, then $p_5\geq 59$. Similar argument given in subcase 1.1 implies $a_2\leq 2$. If $a_2=1$, then $\sigma(5^{2})=31$, it follows that $p_5=31$ but this is absurd since $59\leq p_5$. If, $a_2=2$ then $\sigma(5^{4})=781$, it follows that $p_5=71$. But then
\begin{align*}
  I(N)\geq I(2\cdot5^4\cdot11^2\cdot53^2\cdot71^2)&=\frac{3}{2}\cdot I(5^4\cdot11^2\cdot53^2\cdot71^2)\\
  &=\frac{5840769081}{2742286250}>\frac{21}{10}.  
\end{align*}

Hence this case is impossible.\\

\textbf{Subcase 1.5:}
If $N=2\cdot 5^{2a_2}\cdot11^{2a_3}\cdot59^{2a_4}\cdot p_5^{2a_5}$, then $61\leq p_5\leq 1069$. Now if $p_3=11\mid 
 \sigma(5^{2a_2})$, immediately $71 \mid \sigma(5^{2a_2})$ which means $p_5=71$, but then 
 $$\frac{21}{10}< I(2\cdot 5^2 \cdot11^2\cdot 59^2\cdot 71^2)<  I(2\cdot 5^{2a_2}\cdot11^{2a_3}\cdot59^{2a_4}\cdot 71^{2a_5}).$$ Hence $11 \nmid \sigma(5^{2a_2})$. If $59\mid \sigma(5^{2a_2})$, immediately $35671\mid \sigma(5^{2a_2})$ and it implies that $p_5=35671$ but $ p_5\leq 1069$. Therefore $p_5$ must divide $\sigma(5^{2a_2})$. Note that, if $5\mid \sigma(q^{2a_q})$, then $q\equiv1\pmod{10}$ due to Corollary \ref{cor12}. Therefore, $5 \mid \sigma(11^{2a_3})$ is possible but then $3221 \mid \sigma(11^{2a_3})$, which is impossible. Hence $5$ must divide $\sigma(p_5^{2a_5})$.\\
 Therefore $p_5$ must belong to $\{p: p~\text{is a prime such that}~p\equiv 1 \pmod{10}, 70<p<1062 \}$.\\
Note that, $7 \nmid \sigma(5^{2a_2}),\sigma(59^{2a_4})$. If $7 \mid \sigma(11^{2a_3})$ then immediately $19 \mid \sigma(11^{2a_3})$, which is impossible. Hence $7$ must divide $\sigma(p_5^{2a_5})$. Since $5\mid \sigma(p_5^{2a_5})$ and $7\mid \sigma(p_5^{2a_5})$ we have $p_5\in$  \{71, 151, 191, 211, 281, 331, 401, 421, 431, 491, 541, 571, 631, 641, 701, 751, 821, 911, 991, 1031, 1051, 1061\}. The following table summarizes the subcase when  $p_5\in$  \{71, 151, 191, 211, 281, 331, 401, 421, 431, 491, 541, 571, 631, 641, 701, 751, 821, 911, 991, 1031, 1051, 1061\}.

\begin{center}
\renewcommand{\arraystretch}{1.6} 

\setlength{\tabcolsep}{6pt} 
\begin{tabular}{|>{\raggedright\arraybackslash}p{1.4cm}|
                >{\raggedright\arraybackslash}p{4.5cm}|
                >{\raggedright\arraybackslash}p{8.5cm}|}

\hline
\textbf{Case} & \textbf{Assumption} & \textbf{Contradiction} \\
\hline

1 & $7 \mid \sigma(71^{2a_5})$ & $883 \mid \sigma(71^{2a_5})$, which is impossible. \\
\hline

2 & $7 \mid \sigma(151^{2a_5})$ & Then $3 \mid \sigma(151^{2a_5})$, which is absurd due to Lemma \ref{lemma19}. \\
\hline

3 & $7 \mid \sigma(191^{2a_5})$ & $31 \mid \sigma(191^{2a_5})$, which is impossible. \\
\hline

4 & $7 \mid \sigma(211^{2a_5})$ & $307189 \mid \sigma(211^{2a_5})$, which is impossible. \\
\hline

5 & $7 \mid \sigma(281^{2a_5})$ & $29 \mid \sigma(281^{2a_5})$, which is impossible. \\
\hline

6 & $7 \mid \sigma(331^{2a_5})$ & $3 \mid \sigma(331^{2a_5})$, which is impossible due to Lemma \ref{lemma19}. \\
\hline

7 & $7 \mid \sigma(401^{2a_5})$ & $23029 \mid \sigma(401^{2a_5})$, which is impossible. \\
\hline

8 & $7 \mid \sigma(421^{2a_5})$ & $797310237403261 \mid \sigma(421^{2a_5})$, which is impossible. \\
\hline

9 & $7 \mid \sigma(431^{2a_5})$ & $397 \mid \sigma(431^{2a_5})$, which is impossible. \\
\hline

10 & $7 \mid \sigma(491^{2a_5})$ & $617 \mid \sigma(491^{2a_5})$, which is impossible. \\
\hline

11 & $7 \mid \sigma(541^{2a_5})$ & $3 \mid \sigma(541^{2a_5})$, which is impossible due to Lemma \ref{lemma19}. \\
\hline

12 & $7 \mid \sigma(571^{2a_5})$ & $3 \mid \sigma(571^{2a_5})$, which is impossible due to Lemma \ref{lemma19}. \\
\hline

13 & $7 \mid \sigma(631^{2a_5})$ & $6032531 \mid \sigma(631^{2a_5})$, which is impossible. \\
\hline

14 & $7 \mid \sigma(641^{2a_5})$ & $58789 \mid \sigma(641^{2a_5})$, which is impossible. \\
\hline

15 & $7 \mid \sigma(701^{2a_5})$ & $16975792017452101 \mid \sigma(701^{2a_5})$, which is impossible. \\
\hline

16 & $7 \mid \sigma(751^{2a_5})$ & $3 \mid \sigma(751^{2a_5})$, which is impossible due to Lemma \ref{lemma19}. \\
\hline
17 & $7 \mid \sigma(821^{2a_5})$ & $211 \mid \sigma(821^{2a_5})$, which is impossible. \\
\hline

18 & $7 \mid \sigma(911^{2a_5})$ & $81750272028928231 \mid \sigma(911^{2a_5})$, which is impossible. \\
\hline

19 & $7 \mid \sigma(991^{2a_5})$ & $3 \mid \sigma(991^{2a_5})$, which is impossible due to Lemma \ref{lemma19}. \\
\hline

20 & $7 \mid \sigma(1031^{2a_5})$ & $97 \mid \sigma(1031^{2a_5})$, which is impossible. \\
\hline

21 & $7 \mid \sigma(1051^{2a_5})$ & $29 \mid \sigma(1051^{2a_5})$, which is impossible. \\
\hline

22 & $7 \mid \sigma(1061^{2a_5})$ & $160969 \mid \sigma(1061^{2a_5})$, which is impossible. \\
\hline

\end{tabular}
\vspace{0.5em}
\end{center}
Hence $p_5$ cannot be belong to the set \{71, 151, 191, 211, 281, 331, 401, 421, 431, 491, 541, 571, 631, 641, 701, 751, 821, 911, 991, 1031, 1051, 1061\}, which is a contradiction. Hence this case is also impossible.\\

\textbf{Subcase 1.6:}
If $N=2\cdot 5^{2a_2}\cdot11^{2a_3}\cdot61^{2a_4}\cdot p_5^{2a_5}$, then $67\leq p_5\leq 661$. Now if $5 \mid \sigma(11^{2a_3})$ then immediately  $3221 \mid \sigma(11^{2a_3})$, which is impossible. If $5 \mid \sigma(61^{2a_4})$ then immediately  $21491 \mid \sigma(61^{2a_4})$, which is impossible. Therefore $5 \mid \sigma(p_5^{2a_5})$ also note that, $7 \nmid \sigma(5^{2a_2})$ and $7 \nmid \sigma(61^{2a_4})$ . If $7 \mid \sigma(11^{2a_3})$ then immediately $19 \mid \sigma(11^{2a_3})$, which is impossible. Hence $7\mid\sigma(p_5^{2a_5})$. Therefore, similar argument given in subcase 1.5 proves that, $N$ cannot be a friend of $20$.\\

\textbf{Subcase 1.7:}
 If $N=2\cdot 5^{2a_2}\cdot11^{2a_3}\cdot67^{2a_4}\cdot p_5^{2a_5}$, then $71\leq p_5\leq 331$. If $5 \mid \sigma(11^{2a_3})$ then immediately  $3221 \mid \sigma(11^{2a_3})$, which is impossible. Since $5 \nmid \sigma(67^{2a_4})$, we must have $5 \mid \sigma(p_5^{2a_5})$ and it implies $p_5 \equiv 1 \pmod{10}$. Note that, if $7 \mid \sigma(11^{2a_3})$ then immediately $19 \mid \sigma(11^{2a_3})$, which is impossible. Again if $7 \mid \sigma(67^{2a_4})$ then $3 \mid \sigma(67^{2a_4})$, which is impossible. Hence $7 \mid \sigma(p_5^{2a_5})$. Similar argument given in subcase 1.5 implies that, $N$ cannot be a friend of $20$.\\

\textbf{Subcase 1.8:}
If $N=2\cdot 5^{2a_2}\cdot11^{2a_3}\cdot71^{2a_4}\cdot p_5^{2a_5}$ then $73\leq p_5\leq 257$. If $5 \mid \sigma(11^{2a_3})$ then immediately  $3221 \mid \sigma(11^{2a_3})$, which is impossible.  Now if $5 \mid \sigma(71^{2a_4})$ then immediately  $2221 \mid \sigma(71^{2a_4})$, which is impossible. Therefore $5 \mid \sigma(p_5^{2a_5})$ which implies $p_5 \equiv 1 \pmod{10}$. Note that, $7 \nmid \sigma(5^{2a_2})$ and if $7 \mid \sigma(11^{2a_2})$ then immediately $19 \mid \sigma(11^{2a_3})$, which is impossible. Also if $7 \mid \sigma(71^{2a_4})$ then $883 \mid \sigma(71^{2a_4})$, which is impossible. Hence $7 \mid \sigma(p_5^{2a_5})$. Therefore, similar argument given in subcase 1.5 proves that, $N$ cannot be a friend of $20$.\\

\textbf{Subcase 1.9:}
If $N=2\cdot 5^{2a_2}\cdot11^{2a_3}\cdot73^{2a_4}\cdot p_5^{2a_5}$ then $79\leq p_5\leq 233$. 
If $5 \mid \sigma(11^{2a_3})$ then immediately  $3221 \mid \sigma(11^{2a_3})$, which is impossible. Since $5 \nmid \sigma(73^{2a_4})$, implies $5 \mid \sigma(p_{5}^{2a_5})$ and $p_5 \equiv 1 \pmod{10}$. Note that, $7 \nmid \sigma(5^{2a_2})$,$73^{2a_4}$ and if $7 \mid \sigma(11^{2a_2})$ then immediately $19 \mid \sigma(11^{2a_3})$, which is impossible. Hence $7 \mid \sigma(p_5^{2a_5})$. Therefore, similar argument given in subcase 1.5 proves that, $N$ cannot be a friend of $20$. \\

\textbf{Subcase 1.10:}
If $N=2\cdot 5^{2a_2}\cdot11^{2a_3}\cdot79^{2a_4}\cdot p_5^{2a_5}$, then $83\leq p_5\leq 181$. If $5 \mid \sigma(11^{2a_3})$ then immediately  $3221 \mid \sigma(11^{2a_3})$, which is impossible. Since $5 \nmid \sigma(79^{2a_4})$, implies $5 \mid \sigma(p_{5}^{2a_5})$ and $p_5 \equiv 1 \pmod{10}$. Note that, $7 \nmid \sigma(5^{2a_2})$. If $7 \mid \sigma(11^{2a_2})$ then immediately $19 \mid \sigma(11^{2a_3})$, which is impossible. If $7 \mid \sigma(79^{2a_4})$ then immediately $3 \mid \sigma(79^{2a_4})$, which is impossible. Hence $7 \mid \sigma(p_5^{2a_5})$. Therefore, similar argument given in subcase 1.5 proves that, $N$ cannot be a friend of $20$.

\begin{figure}[h]
    \centering
    \includegraphics[width=1\linewidth]{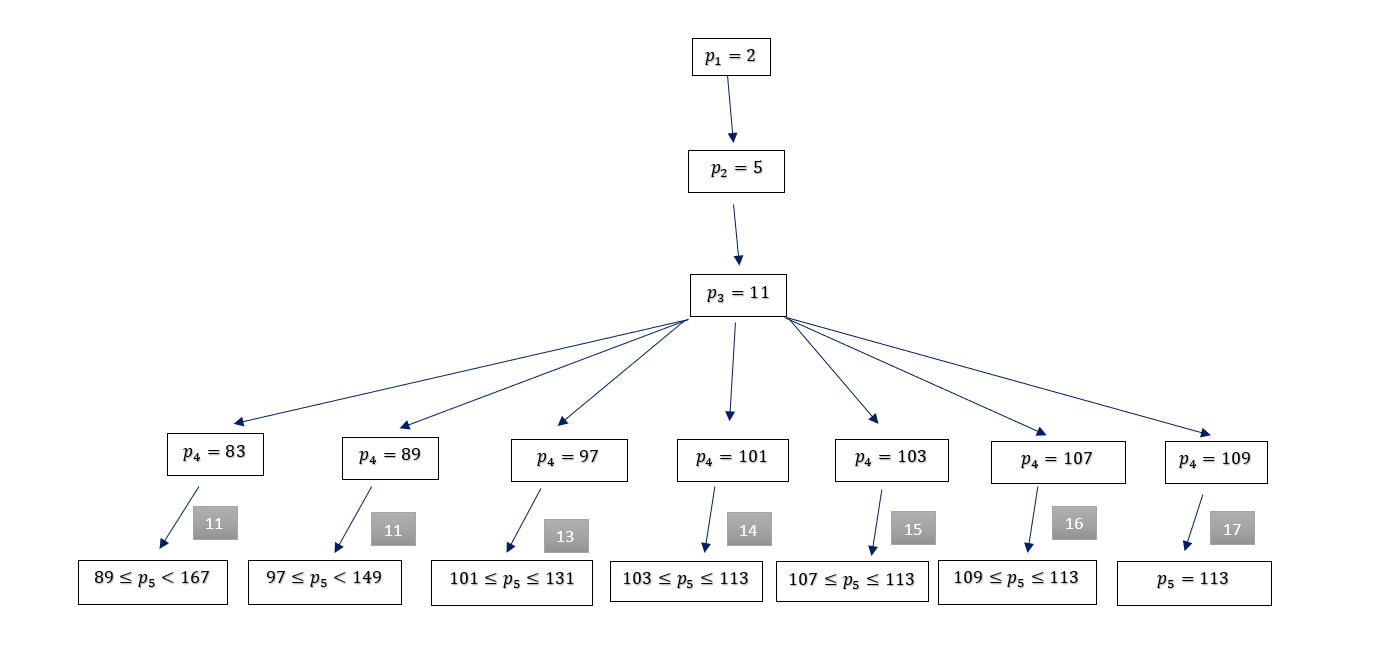}
    \caption{\textcolor{Mycolor1}{Possible subcases when $p_3=11$ (continue)}}
    \label{fig:enter-label}
\end{figure}

\textbf{Subcase 1.11:}
If $N=2\cdot 5^{2a_2}\cdot11^{2a_3}\cdot83^{2a_4}\cdot p_5^{2a_5}$, then $89\leq p_5\leq 167$. If $5 \mid \sigma(11^{2a_3})$ then immediately  $3221 \mid \sigma(11^{2a_3})$, which is impossible. Since $5 \nmid \sigma(83^{2a_4})$, implies $5 \mid \sigma(p_{5}^{2a_5})$ and $p_5 \equiv 1 \pmod{10}$. Note that, $7 \nmid \sigma(5^{2a_2}), \sigma(83^{2a_4})$. If $7 \mid \sigma(11^{2a_2})$ then immediately $19 \mid \sigma(11^{2a_3})$, which is impossible. Hence $7 \mid \sigma(p_5^{2a_5})$. Therefore, similar argument given in subcase 1.5 proves that, $N$ cannot be a friend of $20$.\\

\textbf{Subcase 1.12:}
If $N=2\cdot 5^{2a_2}\cdot11^{2a_3}\cdot89^{2a_4}\cdot p_5^{2a_5}$, then $97\leq p_5\leq 149$. If $5 \mid \sigma(11^{2a_3})$ then immediately  $3221 \mid \sigma(11^{2a_3})$, which is impossible. Since $5 \nmid \sigma(89^{2a_4})$, implies $5 \mid \sigma(p_{5}^{2a_5})$ and $p_5 \equiv 1 \pmod{10}$. Note that, $7 \nmid \sigma(5^{2a_2}), \sigma(89^{2a_4})$. If $7 \mid \sigma(11^{2a_3})$ then immediately $19 \mid \sigma(11^{2a_3})$, which is impossible. Hence $7 \mid \sigma(p_5^{2a_5})$. Therefore, similar argument given in subcase 1.5 proves that, $N$ cannot be a friend of $20$.\\

\textbf{Subcase 1.13:}
If $N=2\cdot 5^{2a_2}\cdot11^{2a_3}\cdot97^{2a_4}\cdot p_4^{2a_5}$, then $101\leq p_5\leq 131$. If $5 \mid \sigma(11^{2a_3})$ then immediately  $3221 \mid \sigma(11^{2a_3})$, which is impossible. Since $5 \nmid \sigma(97^{2a_4})$, implies $5 \mid \sigma(p_{5}^{2a_5})$ and $p_5 \equiv 1 \pmod{10}$. Note that, $7 \nmid \sigma(5^{2a_2}), \sigma(97^{2a_4})$. If $7 \mid \sigma(11^{2a_2})$ then immediately $19 \mid \sigma(11^{2a_3})$, which is impossible. Hence $7 \mid \sigma(p_5^{2a_5})$. But there is no prime $101\leq p_5\leq 131$ such that $5\cdot 7 \mid \sigma(p_5^{2a_5})$. This proves that, $N$ cannot be a friend of $20$.\\

\textbf{Subcase 1.14:}
If $N=2\cdot 5^{2a_2}\cdot11^{2a_3}\cdot101^{2a_4}\cdot p_5^{2a_5}$, then $103\leq p_5\leq 113$. If $5 \mid \sigma(11^{2a_3})$ then immediately  $3221 \mid \sigma(11^{2a_3})$, which is impossible. If $5 \mid \sigma(101^{2a_4})$ then immediately $491 \mid \sigma(101^{2a_4})$, implies $5 \mid \sigma(p_{5}^{2a_5})$ and $p_5 \equiv 1 \pmod{10}$. Note that, $7 \nmid \sigma(5^{2a_2}), \sigma(101^{2a_4})$. If $7 \mid \sigma(11^{2a_2})$ then immediately $19 \mid \sigma(11^{2a_3})$, which is impossible. Hence $7 \mid \sigma(p_5^{2a_5})$. But there is no prime $103\leq p_5\leq 131$ such that $5\cdot 7 \mid \sigma(p_5^{2a_5})$. This proves that, $N$ cannot be a friend of $20$.\\

\textbf{Subcase 1.15:}
If $N=2\cdot 5^{2a_2}\cdot11^{2a_3}\cdot103^{2a_4}\cdot p_5^{2a_5}$, then $107\leq p_5\leq 113$. If $5 \mid \sigma(11^{2a_3})$ then immediately  $3221 \mid \sigma(11^{2a_3})$, which is impossible. Since $5 \nmid \sigma(103^{2a_4})$, implies $5 \mid \sigma(p_{5}^{2a_5})$ and $p_5 \equiv 1 \pmod{10}$. Note that, $7 \nmid \sigma(5^{2a_2}), \sigma(103^{2a_4})$. If $7 \mid \sigma(11^{2a_2})$ then immediately $19 \mid \sigma(11^{2a_3})$ which is impossible. Hence $7 \mid \sigma(p_5^{2a_5})$. There is no prime $107\leq p_5\leq 131$ such that $5\cdot 7 \mid \sigma(p_5^{2a_5})$. This proves, $N$ cannot be a friend of $20$.\\

\textbf{Subcase 1.16:}
If $N=2\cdot 5^{2a_2}\cdot11^{2a_3}\cdot107^{2a_4}\cdot p_5^{2a_5}$, then $109\leq p_5\leq 113$. If $5 \mid \sigma(11^{2a_3})$ then immediately  $3221 \mid \sigma(11^{2a_3})$, which is impossible. Since $5 \nmid \sigma(107^{2a_4})$, implies $5 \mid \sigma(p_{5}^{2a_5})$ and $p_5 \equiv 1 \pmod{10}$. Note that, $7 \nmid \sigma(5^{2a_2})$. If $7 \mid \sigma(11^{2a_2})$ then immediately $19 \mid \sigma(11^{2a_3})$, which is impossible. If $7 \mid \sigma(107^{2a_4})$ then immediately $13 \mid \sigma(107^{2a_4})$, which is impossible. Hence $7 \mid \sigma(p_5^{2a_5})$. But there is no prime $109\leq p_5\leq 131$ such that $5\cdot 7 \mid \sigma(p_5^{2a_5})$. This proves , $N$ cannot be a friend of $20$.\\

\textbf{Subcase 1.17:}
Assume that $N=2\cdot 5^{2a_2}\cdot11^{2a_3}\cdot 109^{2a_4}\cdot113^{2a_5}$. If $5 \mid \sigma(11^{2a_3})$ then immediately  $3221 \mid \sigma(11^{2a_3})$, which is impossible. Since $5 \nmid \sigma(109^{2a_4})$, implies $5 \mid \sigma(113^{2a_5})$, which is absurd since $113 \not \equiv 1 \pmod {10}$. Therefore this proves , $N$ cannot be a friend of $20$\\
\begin{center}
\textbf{\textcolor{RedViolet}{Case 2.}}\\
\textbf{\textcolor{blue}{For $p_3=13$}}\\
\end{center}

If $p_3=13$ then $17 \leq p_4\leq 59$. Note that for $p_4=17,19,23$, we have 
$I(N)>\frac{21}{10}.$
Therefore, we have to consider $29\leq p_4\leq 59$. The following figure describes the possible subcases when $p_3=13$.

\begin{figure}[h]
    \centering
    \includegraphics[width=1\linewidth]{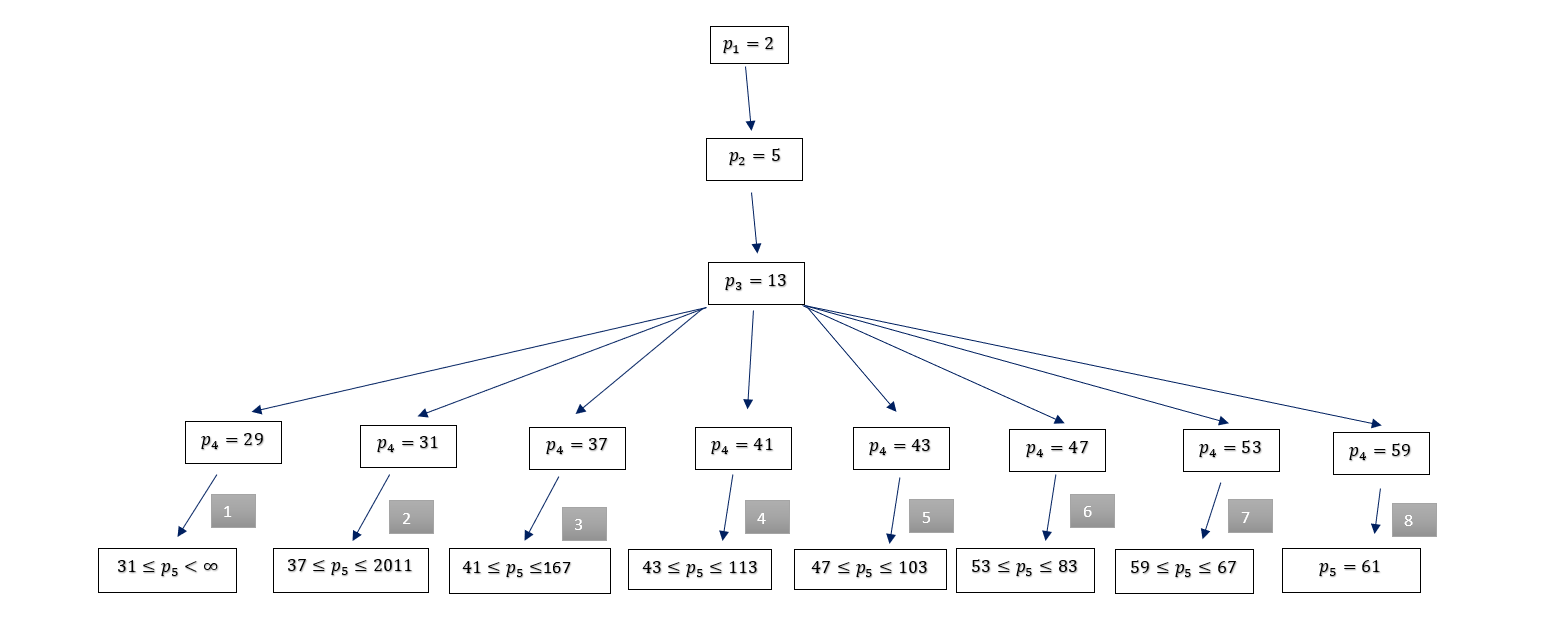}
    \caption{\textcolor{Mycolor1}{Possible subcases when $p_3=13$}}
    \label{fig:enter-label}
\end{figure}

\textbf{Subcase 2.1:}
If $N=2\cdot 5^{2a_2}\cdot13^{2a_3}\cdot29^{2a_4}\cdot p_5^{2a_5}$, then $p_5>29$. Similar argument given in subcase 1.1 forces $a_2=1$. For $a_2=1$, as $\sigma(5^2)=31$, it follows that $p_5=31$. But then 
$$I(N)>\frac{21}{10}.$$
Therefore, this case is impossible.\\

\textbf{Subcase 2.2:} 
If $N=2\cdot 5^{2a_2}\cdot13^{2a_3}\cdot31^{2a_4}\cdot p_5^{2a_5}$, then $37\leq p_5\leq 2011$. Note, $5 \nmid \sigma(13^{2a_3})$. If $5 \mid \sigma(31^{2a_4})$ then immediately $17351 \mid \sigma(31^{2a_4})$, which is impossible as $p_5\leq 2011$. Hence $5 \mid \sigma(p_5^{2a_5})$. Again note that, $13 \nmid \sigma(5^{2a_2})$ and $13 \nmid \sigma(31^{2a_4})$, implies $13 \mid \sigma(p_5^{2a_5})$. Also, $7 \nmid \sigma(5^{2a_2}), \sigma(13^{2a_3}),\sigma(31^{2a_4})$, which implies $7 \mid \sigma(p_5^{2a_5})$. So, we have $5 \cdot 7\cdot 13 \mid \sigma(p_5^{2a_5})$. Hence $p_5 \in$ \{191, 211, 911, 991, 2011\}.
But note $5,13 \nmid \sigma(31^{2a_4})$, implies $p_5 \mid \sigma(31^{2a_4})$. 
Hence $p_5 \in$ \{911, 991, 2011\}. If $911 \mid \sigma(31^{2a_4})$ then immediately $11 \mid \sigma(31^{2a_4})$, which is absurd. If $991 \mid \sigma(31^{2a_4})$ then immediately $3 \mid \sigma(31^{2a_4})$, which is absurd.
If $2011 \mid \sigma(31^{2a_4})$ then immediately $11 \mid \sigma(31^{2a_4})$, which is absurd. Hence this case is impossible.\\

\textbf{Subcase 2.3:}
If $N=2\cdot 5^{2a_2}\cdot13^{2a_3}\cdot37^{2a_4}\cdot p_5^{2a_5}$, then $41\leq p_5\leq 167$. Note that, if $5\mid \sigma(q^{2a_q})$, then $q\equiv1\pmod{10}$ due to Corollary \ref{cor12}, therefore there must exists a prime $p\equiv 1\pmod{10}$, it follows that $p_5\in$\{41, 61, 71, 101, 131, 151\}. But $13, 37, 41, 61 \nmid \sigma(5^{2a_2})$, therefore  $p_5$ must be one of $71, 101, 131, 151$.  But $11\mid \sigma(5^{2a_2})$ whenever $p_5=71, 101, 131, 151\mid \sigma(5^{2a_2})$.
Hence this case is impossible.\\

\textbf{Subcase 2.4:}
If $N=2\cdot 5^{2a_2}\cdot13^{2a_3}\cdot41^{2a_4}\cdot p_5^{2a_5}$, then $43\leq p_5\leq 113$. Since $13,41 \nmid \sigma(5^{2a_2})$, therefore we must that $p_5\mid \sigma(5^{2a_2})$. Therefore $p_5\in \{59, 71, 79, 101, 109\}$. All the cases except when $p_5\in\{59, 79, 109\}$ have already been discussed in subcase 2.3. Now if $p_5=59\mid \sigma(5^{2a_2})$ then $35671\mid \sigma(5^{2a_2})$, if $p_5=79\mid \sigma(5^{2a_2})$, then $31\mid \sigma(5^{2a_2})$ and if $p_5=109\mid \sigma(5^{2a_2})$, then $19\mid \sigma(5^{2a_2})$, therefore, for all the possible values of $p_5$, we get absurd result. Hence this case is impossible.\\

 \textbf{Subcase 2.5:}
If $N=2\cdot 5^{2a_2}\cdot13^{2a_3}\cdot43^{2a_4}\cdot p_5^{2a_5}$, then $47\leq p_5\leq 103$. Since there must exists a prime $p\equiv 1\pmod{10}$, it follows that $p_5 \in$ \{61, 71, 101\}. But $13, 43 \nmid \sigma(5^{2a_2})$, therefore we must that $p_5\mid \sigma(5^{2a_2})$. Since $p_5=61\nmid \sigma(5^{2a_2})$, we have 
 $p_5=71~\text{or,}~101\mid \sigma(5^{2a_2})$. By essentially
identical logic to subcase 2.3, we conclude that this case is impossible.\\

\textbf{Subcase 2.6:}
If $N=2\cdot 5^{2a_2}\cdot13^{2a_3}\cdot47^{2a_4}\cdot p_5^{2a_5}$, then $53\leq p_5\leq 83$. Since there must exists a prime $p\equiv 1\pmod{10}$, it follows that $p_5\in$\{61,71\}. By essentially
identical logic to subcase 2.3, we conclude that this case is impossible.\\

\textbf{Subcase 2.7:}
If $N=2\cdot 5^{2a_2}\cdot13^{2a_3}\cdot53^{2a_4}\cdot p_5^{2a_5}$, then $59\leq p_5\leq 67$. Since there must exists a prime $p\equiv 1\pmod{10}$, it follows that $p_5=61$. But $13, 53, 61\nmid \sigma(5^{2a_2})$. Hence this case is impossible.\\

\textbf{Subcase 2.8:}
Assume that $N=2\cdot 5^{2a_2}\cdot13^{2a_3}\cdot59^{2a_4}\cdot 61^{2a_5}$. Since $13,61\nmid \sigma(5^{2a_2}) $, it follows that $59\mid \sigma(5^{2a_2})$ but it forces $35671\mid \sigma(5^{2a_2})$. Therefore, this case is impossible.\\

\begin{center}
  \textbf{\textcolor{RedViolet}{
  Case 3.}}  \\
  \textbf{\textcolor{blue}{For $p_3=17$}}\\
\end{center}

If $p_3=17$ then $19 \leq p_4\leq 31$. The following figure describes the possible subcases when $p_3=17$.

\begin{figure}[h]
    \centering
    \includegraphics[width=1\linewidth]{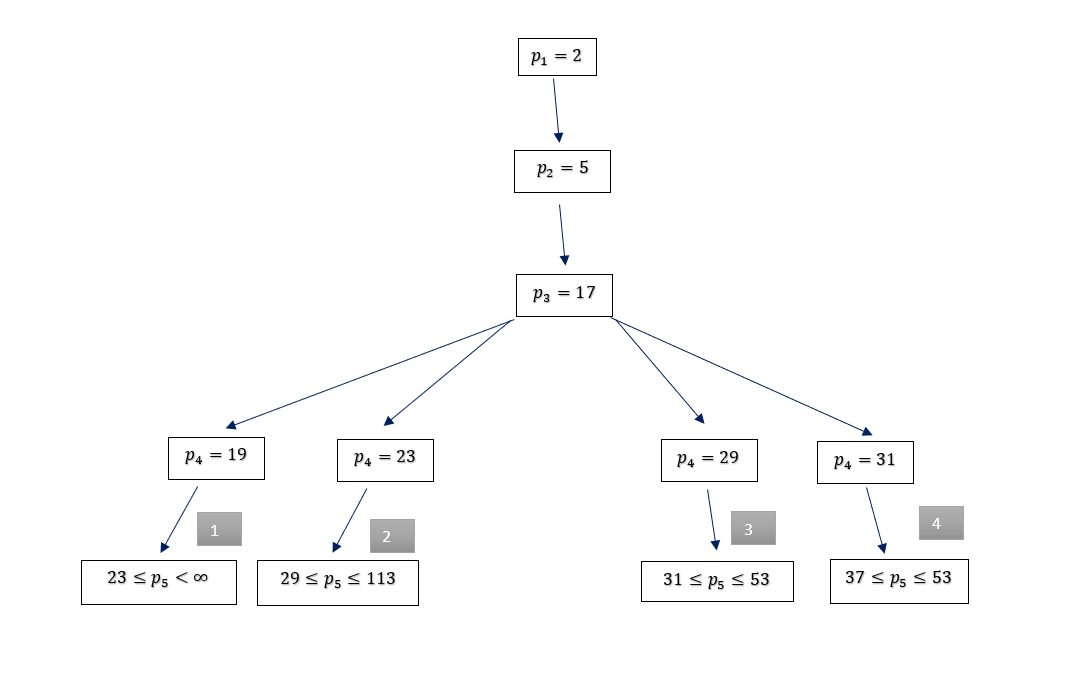}
    \caption{\textcolor{Mycolor1}{Possible subcases when $p_3=17$}}
    \label{fig:enter-label}
\end{figure}

\textbf{Subcase 3.1:}
If $N=2\cdot 5^{2a_2}\cdot17^{2a_3}\cdot19^{2a_4}\cdot p_5^{2a_5}$, then $p_5>19$. Similar argument given in subcase 1.1 forces $a_2=1$. For $a_2=1$, as $\sigma(5^2)=31$, it follows that $p_5=31$. But then 
$$I(N)>\frac{21}{10}.$$
Therefore, this case is impossible.\\

\textbf{Subcase 3.2:}
If $N=2\cdot 5^{2a_2}\cdot17^{2a_3}\cdot23^{2a_4}\cdot p_5^{2a_5}$, then $29\leq p_5\leq 113$. Since there must exists a prime $p\equiv 1\pmod{10}$, it follows that $p_5\in$ \{31, 41, 61, 71, 101\}. But for $p_5=31,41$ we have
$$I(N)>\frac{21}{10}.$$
Therefore $p_5\in$\{61, 71, 101\}. Since  $17,23\nmid \sigma(5^{2a_2})$ , we must have $p_5\mid \sigma(5^{2a_2})$. Therefore $p_5$ can be $71$ or $101$. Note that $11\mid \sigma(5^{2a_2})$ whenever $101\mid \sigma(5^{2a_2})$ or $71\mid \sigma(5^{2a_2})$. Therefore, this case is impossible.\\

\textbf{Subcase 3.3:}
If $N=2\cdot 5^{2a_2}\cdot17^{2a_3}\cdot29^{2a_4}\cdot p_5^{2a_5}$, then $31\leq p_5\leq 53$. Note that $17\mid \sigma(N)$, but $17\nmid 5^{2a_2}, 29^{2a_4}, p_5^{2a_5}$, for any $31\leq p_5\leq 53$. Hence this case is impossible.\\

\textbf{Subcase 3.4:}
If $N=2\cdot 5^{2a_2}\cdot17^{2a_3}\cdot31^{2a_4}\cdot p_5^{2a_5}$, then $37\leq p_5\leq 53$. This subcase is also impossible by the similar argument given in subcase 3.3.\\

\begin{center}
\textbf{\textcolor{RedViolet}{Case 4.}} \\
\textbf{\textcolor{blue}{For $p_3=19$}}\\
\end{center}
If $p_3=19$ then $23 \leq p_4\leq 31$. The following figure describes the possible subcases when $p_3=19$.
\begin{figure}[h]
    \centering
    \includegraphics[width=0.65\linewidth]{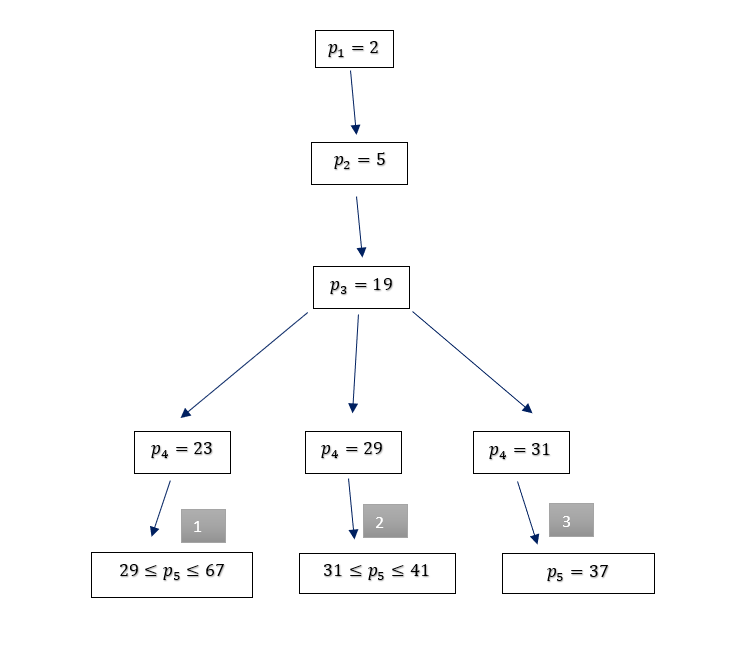}
    \caption{\textcolor{Mycolor1}{Possible subcases when $p_3=19$}}
    \label{fig:enter-label}
\end{figure}

\textbf{Subcase 4.1:}
If $N=2\cdot 5^{2a_2}\cdot19^{2a_3}\cdot23^{2a_4}\cdot p_5^{2a_5}$, then $29\leq p_5\leq 67$. Since there must exists a prime $p\equiv 1\pmod{10}$, it follows that $p_5\in$ \{31, 41, 61\}. If $19\mid \sigma(5^{2a_2})$, then $829\mid \sigma(5^{2a_2})$, therefore $19\nmid  \sigma(5^{2a_2})$. As $23\nmid  \sigma(5^{2a_2})$, it follows that $p_5\mid  \sigma(5^{2a_2})$. Note that, $p_5=31$ as for any other $p_5=41, 61$, $19, 23, p_5\nmid  \sigma(5^{2a_2})$. Note that, $7\mid  \sigma(23^{2a_4})$ only but it forces, $ 79\mid \sigma(23^{2a_4})$. Therefore, this case is also impossible.\\

\textbf{Subcase 4.2:}
If $N=2\cdot 5^{2a_2}\cdot19^{2a_3}\cdot29^{2a_4}\cdot p_5^{2a_5}$, then $31\leq p_5\leq 41$. Since there must exists a prime $p\equiv 1\pmod{10}$, it follows that $p_5\in$\{31, 41\}. If $19\mid \sigma(5^2a_2)$, then $829\mid \sigma(5^{2a_2})$, therefore $19\nmid  \sigma(5^{2a_2})$. As $29\nmid  \sigma(5^{2a_2})$, it follows that $p_5\mid  \sigma(5^{2a_2})$. Note that $p_5=31$ as for $p_5=41$, $19, 23, p_5\nmid  \sigma(5^{2a_2})$. Since, $7\mid  \sigma(29^{2a_4})$ only but it forces $ 88009573 \mid \sigma(29^{2a_4})$. Hence this case is impossible.\\

\textbf{Subcase 4.3:}
Assume that $N=2\cdot 5^{2a_2}\cdot19^{2a_3}\cdot31^{2a_4}\cdot 37^{2a_4}$. Note that $7\mid  \sigma(37^{2a_5})$ only but it forces $67\mid \sigma(37^{2a_5})$. Therefore, this case is impossible.\\

This proves that any friend $N$ of $20$ has at least six distinct prime divisors.
\end{proof}

The proof of the Theorem  \ref{thm 1.1} now follows from Lemma \ref{nlemma}, Lemma \ref{lemma17} and Lemma \ref{lemma19}.

    \subsection{Proof of Theorem \ref{thm 1.2}.}
We begin with the following lemmas:
\begin{lemma}\label{lem3.3}
    Let a function $\psi: [1,\infty)	\rightarrow\mathbb{R}$ defined by $\psi(x)=ax-a^x$ then $\psi$ is a strictly decreasing function of $x$ in $[1,\infty)$ for all $a>e$.
\end{lemma}
\begin{proof}
     Let $\psi: [1,\infty)	\rightarrow\mathbb{R}$ is defined by $\psi(x)=ax-a^x$, then clearly
     \begin{align*}
         \psi'(x)=a-a^{x} \log a < 0; ~~~~\forall x \in [1,\infty)~ \text{and}~ \forall a > e.
            \end{align*}
            Hence, $\psi$ is a strictly decreasing function of $x$ in $[1,\infty)$ for all $a>e$.
\end{proof}
\begin{lemma}\label{lem3.4}
    Let $(c_{1}, c_{2}, ... , c_{k})
    $ be any partition of $n$ i.e; $\sum_{i=1}^{k} c_{i}=n$ for $1 \leq k \leq n$ and all $c_{i} \in \mathbb{N}$, then for any integer $a>e$ we have $$an < \sum_{i=1}^{k} a^{c_i}.$$
  
\end{lemma}
\begin{proof}
By Lemma \ref{lem3.3}, for each $c_{i} \geq 1$ and for each $a > e$, we have
\begin{align}\label{3.3}
ac_{i}<a^{c_i}.
\end{align}
Since $(c_{1}, c_{2}, ... , c_{k})
    $ is a partition of $n$, we have $\displaystyle{
    a\sum_{i=1}^{k}c_{i}=an}$. Now from (\ref{3.3}), we have
    \begin{align*}
    a\sum_{i=1}^{k}c_{i}=an < \sum_{i=1}^{k}a^{c_{i}}.
    \end{align*}
    This completes the prove.
\end{proof}

 \begin{lemma}\label{o1}
The minimum element in the set $\mathcal{H}_{2a-1,5}$ is $8a-4$.
  
\end{lemma}
\begin{proof}
   If $\displaystyle{\sum_{i=1}^{k}c_{i}=2a-1}$, then from Lemma \ref{lem3.4}, we get
   \begin{align}\label{0}
   5(2a-1)<\sum_{i=1}^{k}5^{c_i}.
   \end{align}
   Now since $k \leq 2a-1$, from (\ref{0}) we have $\displaystyle{5(2a-1)-(2a-1)<\sum_{i=1}^{k}5^{c_i}-k}$. Thus, 
   
    \begin{align*}
   4(2a-1)<\sum_{i=1}^{k}5^{c_i}-k.
   \end{align*}
   This completes the prove.
\end{proof}
\begin{remark}
The minimum element in the set $\mathcal{H}_{2a-1,5}$ is from the set $\mathcal{A}_{2a-1,5}(2a-1)$.
\end{remark}

Now we have enough materials to prove the theorem. Let $N=2\cdot5^{2a}\cdot \displaystyle{\prod_{i=1}^{s-2}p_{i}^{2\gamma_{i}}}$ be a friend of 20 with $\omega(N)=s$, then $\frac{\sigma(N)}{N}=\frac{21}{10}$
 which implies,
 \begin{align}\label{o4.2}
 \sigma(5^{2a})\prod_{i=1}^{s-2}\sigma(p_{i}^{2\gamma_{i}})=7\cdot 5^{2a-1} \prod_{i=1}^{s-2}p_{i}^{2\gamma_{i}}.
 \end{align}\\
 In the right hand side of (\ref{o4.2}), there are $2a-1$ numbers of 5. Now we consider
 the problem case by case.
 
\textbf{\textcolor{RedViolet}{Case 1.}}
There is only one prime divisor of $N$ (without loss of generality, we may assume $p_{1}$) such that $5^{2a-1}~\|~\sigma(p_{1}^{2\gamma_{1}})$ i.e; $|\mathcal{F}_5(N)|=1$ and $p_1 \equiv 1 \pmod {10}$. Clearly, $v_{5}(\sigma(p_{1}^{2\gamma_{1}}))=2a-1$ where $2\gamma_{1}+1 \equiv 0 \pmod {5^{2a-1}}.$

   \textbf{\textcolor{RedViolet}{Case 2.}}
    There are two prime divisors of $N$ (without loss of generality, we may assume $p_{1}, p_{2}$) such that $5^{2a-1}~\|~\sigma(p_{1}^{2\gamma_{1}})\sigma(p_{2}^{2\gamma_{2}})$ i.e; $|\mathcal{F}_5(N)|=2$ and $p_1 \equiv p_2 \equiv 1 \pmod {10}$. Clearly, $v_{5}(\sigma(p_{1}^{2\gamma_{1}}))+ v_{5}(\sigma(p_{2}^{2\gamma_{2}}))=2a-1$, where $2\gamma_{i}+1 \equiv 0 \pmod {5^{v_{5}(\sigma(p_{i}^{2\gamma_{i}}))}}$ for $i=1, 2$ . 
   
   Continuing like this manner we have case $2a-2$.
   
\textbf{\textcolor{RedViolet}{Case 2a-2.}} There are $2a-2$ numbers of prime divisors of $N$ (without loss of generality, we may assume $p_{1}, p_{2}$, $p_{3}$,..., $p_{2a-2}$) such that $$5^{2a-1}~\|~\sigma(p_{1}^{2\gamma_{1}})\sigma(p_{2}^{2\gamma_{2}})\cdots \sigma(p_{2a-2}^{2\gamma_{2a-2}})$$ i.e; $|\mathcal{F}_5(N)|=2a-2$ and $p_1 \equiv p_2 \equiv \cdots \equiv p_{2a-2} \equiv 1 \pmod {10}$. Clearly, $\displaystyle{\sum_{i=1}^{2a-2}v_{5}(\sigma(p_{i}^{2\gamma_{i}}))=2a-1}$, where $2\gamma_{i}+1 \equiv 0 \pmod {5^{v_{5}(\sigma(p_{i}^{2\gamma_{i}}))}}$ for $i=1, 2,..., 2a-2 $ . In this case, all $v_{5}(\sigma(p_{k_i}^{2\gamma_{i}}))=1$ except one $v_{5}(\sigma(p_{j}^{2\gamma_{j}}))$ which is 2.

      \textbf{\textcolor{RedViolet}{Case 2a-1.}} There are exactly  $2a-1$ numbers of prime divisors of $N$ (without loss of generality, we may assume $p_{1}$, $p_{2}$, $p_{3}$,..., $p_{2a-1}$) such that $$5^{2a-1}~\|~\sigma(p_{1}^{2\gamma_{1}})\sigma(p_{2}^{2\gamma_{2}})\cdots \sigma(p_{2a-1}^{2\gamma_{2a-1}})$$ i.e; $|\mathcal{F}_5(N)|=2a-1$ and $p_1 \equiv p_2 \equiv \cdots \equiv p_{2a-1} \equiv 1 \pmod {10}$. Clearly, $\displaystyle{\sum_{i=1}^{2a-1}v_{5}(\sigma(p_{i}^{2\gamma_{i}}))=2a-1}$, where $2\gamma_{i}+1 \equiv 0 \pmod {5^{v_{5}(\sigma(p_{i}^{2\gamma_{i}}))}}$ for $i=1, 2,..., 2a-1$. In this case, all $v_{5}(\sigma(p_{i}^{2\gamma_{i}}))=1.$
      
If we consider all the cases mentioned above and choose the minimum value of $\displaystyle{\sum_{i=1}^{t}2\gamma_{i}}$ such that $5^{2a-1}~\|~\sigma(p_{1}^{2\gamma_{1}})\sigma(p_{2}^{2\gamma_{2}})\sigma(p_{3}^{2\gamma_{3}})\cdots \sigma(p_{t}^{2\gamma_{t}})$ satisfying $\displaystyle{\sum_{i=1}^{t}v_{5}(\sigma(p_{i}^{2\gamma_{i}}))=2a-1}$, for $1\leq t \leq |\mathcal{F}_5(N)|=2a-1$ and while excluding the primes with exponent 2 that are not included in those cases (except 2, since 2 has exponent 1), we can obtain the minimum possible value for $\Omega(N)$.

To find the minimum value of $ \displaystyle{\sum_{i=1}^{t}2\gamma_{i}}$, it is enough to consider  $2\gamma_j=5^{v_{5}(\sigma(p_{j}^{2\gamma_{j}}))}-1$  since $2 \gamma_{i}+1 \equiv 0 \pmod{v_{5}(\sigma(p_{i}^{2\gamma_{i}}))}$, for $j=1, 2,..., t$.

A careful calculation gives the minimum value as following: 
\begin{align}\label{6}
\sum_{i=1}^{t}2\gamma_{i}=\sum_{i=1}^{t}\bigg(5^{v_{5}(\sigma(p_{i}^{2\gamma_{i}}))}-1\bigg)=\sum_{i=1}^{t}5^{v_{5}(\sigma(p_{i}^{2\gamma_{i}}))}-t.
\end{align}
 Now, the sum in (\ref{6}) is same as the minimum element in the set $\mathcal{H}_{2a-1,5}$.
 
  From Lemma \ref{o1} we know that the minimum element in the set $\mathcal{H}_{2a-1,5}$ is $8a-4$. Hence, the minimum value of $\Omega(N)$ is $2\omega(N)+6a-5$ i.e.; $\Omega(N) \geq 2\omega(N)+6a-5$. This completes the proof.
  
\begin{remark}
Since we have $\Omega(N) \geq 2\omega(N)+6a-5$, now using the  completely additive property of $\Omega(N)$ we have $\Omega(2)+\Omega(5^{2a})+\Omega(m^2)=1+2a+ 2 \Omega(m)\geq 2\omega(N)+6a-5$ i.e; 
\begin{align}\label{ap}
\Omega(m) \geq \omega(N)+ 2a-3.
\end{align}
Finally, using additive property of $\omega(N)$and from (\ref{ap}) we get,

\begin{align}
\Omega(m) \geq \omega(m)+ 2a-1.
\end{align}
\end{remark}
\subsection{Proof of corollary \ref{coro 1.3}.}
Since $N=2\cdot 5^{2a}m^2$ is a friend of $20$, $N/2=5^{2a}m^2$ is an odd integer with abundancy index $7/5$, so, it is an odd $7/5$-perfect number and since $\Omega(m) \leq K$, from (\ref{ap}) we get, $K-2a+3 \geq \omega(N)$. Now using Lemma \ref{pn}, we have 
$N/2 < 5\cdot 6^{{(2^{{\omega(N)}}-1})^2}< 5\cdot 6^{(2^{K-2a+3}-1)^2}$, implies $N < 10\cdot 6^{(2^{K-2a+3}-1)^2}$. This completes the proof.

\subsection{Proof of Theorem \ref{thm 1.4}.}
Let $N$ be a friend of $20$ and also, let that exponent of $5$ in the prime factorization of friend $N$ of $20$ is congruent to $-1$ modulo $f$. Therefore, if $v_5(N)=2a$, then $2a\equiv -1\pmod f$ that is, $2a+1\equiv 0 \pmod f$ where $
 5^{f}\equiv 1\pmod p.$
Thus $p\mid \sigma(5^{2a})$ follows from Lemma \ref{ss2024thm1.1}. Therefore, assume that $v_p(N)=2a_p$. If  $2a_{p}\not \equiv -1 \pmod{f}$, then the theorem holds. Suppose that  $2a_{p} \equiv -1 \pmod{f}$ that is,  $2a_{p}+1 \equiv 0 \pmod{f}$. We are given that $p\equiv1\pmod 6$ therefore
$$\sigma(p^{2a_p}) \equiv 1+1+\cdots +1=2a_{p}+1 \pmod{3}.$$
Since $3\mid f$ we have that $2a_{p}+1\equiv 0 \pmod{3}$ this implies that 
$\sigma(p^{2a_p})\equiv 0\pmod 3.$
It follows that $3$ is a divisor of $N$ but this is impossible due to Lemma \ref{lemma19}. Therefore, we must have $2a_p\not\equiv -1\pmod{f}$. This proves the theorem.
\subsection{Proof of Corollary \ref{coro 1.6}.}
Let $N=2F^2$ be a friend of $20$. If possible, assume that $F$ is square-free, then since the order of $5$ in $(\mathbb{Z}/31\mathbb{Z})^\times$ is $3$ and $v_5(N)=2\equiv -1 \pmod{3}$ by Theorem \ref{thm 1.4} there exists an odd prime divisor of $N$ (say q) such that $v_{q}(N)\not\equiv 2 \pmod 3$ but $v_{q}(N)=2$ which is a contradiction. Therefore, our assumption that $F$ is square-free is wrong and thus $F$ must be a non square-free positive integer. This completes the proof.

\subsection{Proof of Theorem \ref{thm 1.8}}
Let 
$$N=2\cdot 5^{2a}\cdot \prod_{3\leq i \leq \omega(N)}q^{2a_i}_i$$ 
be a friend of 20, such that no $q_i=3,7$. To prove this theorem, it suffices to show that $q_r$ must be strictly less than $p_{L}$ by Lemma \ref{lemma13}. Suppose that $q_r\geq p_{L}$ where
$$L=\lceil\frac{\mathcal{U}\omega(N)}{\mathcal{V}}\rceil,~~~\frac{\mathcal{U}}{\mathcal{V}} > \frac{1}{\frac{28}{25}\cdot \displaystyle{\prod_{5\leq i\leq r+1} (1-\frac{1}{p_i})-1}}~~~(\text{where}~p_i~\text{is the $i$-th prime number}),$$
$\frac{\mathcal{U}}{\mathcal{V}}\in \mathbb{Q}^{+}\setminus \mathbb{Z}^{+}$, and $(\mathcal{U},\mathcal{V})=1$  such that $\mathcal{UV}(r-2)+2\mathcal{U}+\mathcal{V}>\mathcal{V}^2$. Note that $\mathcal{U}>\mathcal{V}>1$. Using Property (\ref{p4}) and Property (\ref{p5}) we get,
\begin{align*}
    I(N) & \leq I\left(2\cdot 5^{2a} \cdot p_5^{2a_3} \cdot p_6^{2a_4} \dots \cdot p_{r+1}^{2a_{r-1}} \cdot \prod_{r\leq i\leq \omega(N)} p^{2a_i}_{ L+i-r}\right)\\
    &<\frac{3}{2}\cdot \frac{5}{4}\cdot \prod_{5 \leq j\leq (r+1)} \frac{p_j}{p_j-1} \cdot \prod_{r \leq i\leq \omega(N)}\frac{p_{L+i-r}}{p_{L+i-r}-1}.
\end{align*}
This implies that, 
\begin{align*}
    I(N)&<\frac{3}{2}\cdot \frac{5}{4}\cdot \prod_{5 \leq j\leq (r+1)} \frac{p_j}{p_j-1} \cdot \prod_{r \leq i\leq \omega(N)}\frac{L+i-r}{L+i-r-1}\\
    &=\frac{3}{2}\cdot \frac{5}{4}\cdot \prod_{5 \leq j\leq (r+1)} \frac{p_j}{p_j-1} \cdot \frac{L+\omega(N)-r}{L-1}.
\end{align*}

Now we shall show that for any $\omega(N)\in\mathbb{Z^+}$, the following holds,
\begin{align*}
   \frac{L+\omega(N)-r}{L-1}< \frac{\mathcal{U+V}}{\mathcal{U}}.
\end{align*}
Since $\frac{\mathcal{U}}{\mathcal{V}}\in \mathbb{Q}^{+}\setminus \mathbb{Z}^{+}$, we can write $\mathcal{U}=\mathcal{V}k+\delta$, where $k,\delta\in \mathbb{Z^+}$ and $\delta\in \{1,\dots, \mathcal{V}-1\}$. Note that, $(\mathcal{V},\delta)=1$ as $(\mathcal{U,V})=1$. Then we get,
\begin{align*}
\frac{L+\omega(N)-r}{L-1}=\frac{(k+1)\omega(N)-r+\lceil\frac{\delta\omega(N)}{\mathcal{V}}\rceil}{k\omega(N)-1+\lceil\frac{\delta \omega(N)}{\mathcal{V}}\rceil}.    
\end{align*}
We now consider the following cases, where we essentially observe the behavior of $$\frac{(k+1)\omega(N)-r+\lceil\frac{\delta\omega(N)}{\mathcal{V}}\rceil}{k\omega(N)-1+\lceil\frac{\delta \omega(N)}{\mathcal{V}}\rceil},$$ based on the divisibility of $\omega(N)$ by $\mathcal{V}$.\\
\begin{center}
  \textbf{\textcolor{RedViolet}{
  Case 1.}}   
\end{center}
If $\mathcal{V} \nmid \omega(N)$ then $\frac{\delta \omega(N)}{\mathcal{V}}\not\in \mathbb{Z^+}$, therefore we have,
\begin{align*}
  \frac{(k+1)\omega(N)-r+\lceil\frac{\delta\omega(N)}{\mathcal{V}}\rceil}{k\omega(N)-1+\lceil\frac{\delta \omega(N)}{\mathcal{V}}\rceil} &=\frac{(k+1)\omega(N)-r+1+\frac{\delta\omega(N)}{\mathcal{V}}-\{\frac{\delta\omega(N)}{\mathcal{V}}\}}{k\omega(N)+\frac{\delta \omega(N)}{\mathcal{V}}-\{\frac{\delta\omega(N)}{\mathcal{V}}\}}\\
  &=\frac{(\mathcal{V}k+\mathcal{V}+\delta)\omega(N)+(1-r)\mathcal{V}-\mathcal{V}\{\frac{\delta\omega(N)}{\mathcal{V}}\}}{(\mathcal{V}k+\delta)\omega(N)-\mathcal{V}\{\frac{\delta\omega(N)}{\mathcal{V}}\}}\\
  &=\frac{(\mathcal{U}+\mathcal{V})\omega(N)+(1-r)\mathcal{V}-\mathcal{V}\{\frac{\delta\omega(N)}{\mathcal{V}}\}}{\mathcal{U}\omega(N)-\mathcal{V}\{\frac{\delta\omega(N)}{\mathcal{V}}\}}.
    \end{align*}
Note that for any positive integer  $Q$ which is not divisible by $\mathcal{V}$ can be written in the form $Q=\mathcal{V}q+v$, where $q\in \mathbb{Z}_{\geq 0}$ and $v\in \{1,2,\dots,\mathcal{V}-1\}$. Therefore in particular for $Q=\delta\omega(N)$, we have $\{\frac{\delta\omega(N)} {\mathcal{V}}\}\in \{\frac{1}{\mathcal{V}},\frac{2}{\mathcal{V}},\dots, \frac{\mathcal{V}-1}{\mathcal{V}}\}$ and thus,

\begin{align*}
    1\leq \mathcal{V}\{\frac{\delta\omega(N)}{\mathcal{V}}\}\leq \mathcal{V}-1
\end{align*}
that is,
\begin{align}\label{3a}
    (\mathcal{U}+\mathcal{V})\omega(N)+(1-r)\mathcal{V}-\mathcal{V}\{\frac{\delta\omega(N)}{\mathcal{V}}\}\leq (\mathcal{U}+\mathcal{V})\omega(N)+(1-r)\mathcal{V}-1
\end{align}
    and
    \begin{align}\label{3b}
        \mathcal{U}\omega(N)-\mathcal{V}+1\leq \mathcal{U}\omega(N)-\mathcal{V}\{\frac{\delta\omega(N)}{\mathcal{V}}\}.
    \end{align}

Using (\ref{3a}) and (\ref{3b}) we finally get
\begin{align*}
   \frac{(\mathcal{U}+\mathcal{V})\omega(N)+(1-r)\mathcal{V}-\mathcal{V}\{\frac{\delta\omega(N)}{\mathcal{V}}\}}{\mathcal{U}\omega(N)-\mathcal{V}\{\frac{\delta\omega(N)}{\mathcal{V}}\}}\leq \frac{(\mathcal{U}+\mathcal{V})\omega(N)+(1-r)\mathcal{V}-1}{\mathcal{U}\omega(N)-\mathcal{V}+1}.
\end{align*}
Define $\phi: [1,\infty)	\rightarrow\mathbb{R}$ by $\phi(t)=\frac{(\mathcal{U}+\mathcal{V})t+(1-r)\mathcal{V}-1}{\mathcal{U}t-\mathcal{V}+1}$. Note that
$\phi'(t)=\frac{\mathcal{UV}(r-2)+2\mathcal{U}+\mathcal{V}-\mathcal{V}^2}{(\mathcal{U}t-\mathcal{V}+1)^2}$ as $\mathcal{UV}(r-2)+2\mathcal{U}+\mathcal{V}>\mathcal{V}^2$, we have $\phi'(t)>0$ and thus $\phi$ is a strictly increasing function of t in $[1,\infty)$. Since $\lim_{t\rightarrow \infty} \phi(t)=\frac{\mathcal{U+V}}{\mathcal{U}}$ we have $\phi(t)<\frac{\mathcal{U+V}}{\mathcal{U}}$ for all $t\in [1,\infty)$. In particular for $t=\omega(N)$ we have
\begin{align*}
\frac{(\mathcal{U}+\mathcal{V})\omega(N)+(1-r)\mathcal{V}-1}{\mathcal{U}\omega(N)-\mathcal{V}+1}<\frac{\mathcal{U+V}}{\mathcal{U}},
\end{align*}
which immediately implies that 
$$\frac{L+\omega(N)-r}{L-1}< \frac{\mathcal{U+V}}{\mathcal{U}}.$$
\begin{center}
  \textbf{\textcolor{RedViolet}{
  Case 2.}}   
\end{center}
Now if $\mathcal{V}\mid \omega(N)$ then $\frac{\delta \omega(N)}{\mathcal{V}}\in \mathbb{Z^+}$. Therefore,
\begin{align*}
 \frac{(k+1)\omega(N)-r+\lceil\frac{\delta\omega(N)}{\mathcal{V}}\rceil}{k\omega(N)-1+\lceil\frac{\delta \omega(N)}{\mathcal{V}}\rceil}=\frac{(k+1)\omega(N)-r+\frac{\delta\omega(N)}{\mathcal{V}}}{k\omega(N)-1+\frac{\delta \omega(N)}{\mathcal{V}}}&=\frac{(\mathcal{V}k+\delta+\mathcal{V})\omega(N)-\mathcal{V}r}{(\mathcal{V}k+\delta)\omega(N)-\mathcal{V}}\\
 &=\frac{(\mathcal{U}+\mathcal{V})\omega(N)-\mathcal{V}r}{\mathcal{U}\omega(N)-\mathcal{V}}.
 \end{align*}
Define $\tau: [1,\infty)	\rightarrow\mathbb{R}$ by $\tau(t)=\frac{(\mathcal{U}+\mathcal{V})t-\mathcal{V}r}{\mathcal{U}t-\mathcal{V}}$. Then 
$\tau'(t)=\frac{\mathcal{UV}r-\mathcal{UV}-\mathcal{V}^2}{(\mathcal{U}t-\mathcal{V})^2}$ as $r\geq 3$ and $\mathcal{U}>\mathcal{V}$, $\mathcal{UV}(r-1)-\mathcal{V}^2> \mathcal{UV}(r-1)-\mathcal{UV}=\mathcal{UV}(r-2)> 0$ it follows that $\tau'(t)>0$ and thus $\tau$ is a strictly increasing function of t in $[1,\infty)$. Since $\lim_{t\rightarrow \infty} \tau(t)=\frac{\mathcal{U}+\mathcal{V}}{\mathcal{U}}$ we have $\tau(t)<\frac{\mathcal{U+V}}{\mathcal{U}}$ for all $t\in [1,\infty)$. In particular for $t=\omega(N)$ we have
\begin{align*}
\frac{(\mathcal{U}+\mathcal{V})\omega(N)-\mathcal{V}r}{\mathcal{U}\omega(N)-\mathcal{V}}<\frac{\mathcal{U+V}}{\mathcal{U}} 
\end{align*}
which immediately implies that 
\begin{align*}
   \frac{L+\omega(N)-r}{L-1}< \frac{\mathcal{U+V}}{\mathcal{U}}.
\end{align*}
Therefore for any $\omega(N)\in \mathbb{Z^+}$ we have
\begin{align*}
  \frac{L+\omega(N)-r}{L-1}< \frac{\mathcal{U+V}}{\mathcal{U}}.
\end{align*}
which shows that
\begin{align*}
    I(N)&<\frac{3}{2}\cdot \frac{5}{4}\cdot \prod_{5 \leq j\leq (r+1)} \frac{p_j}{p_j-1} \cdot \frac{\mathcal{U+V}}{\mathcal{U}}=\frac{3}{2}\cdot \frac{5}{4}\cdot \prod_{5 \leq j\leq (r+1)} \frac{p_j}{p_j-1}\cdot(1+ \frac{\mathcal{V}}{\mathcal{U}}),
\end{align*}
since $$\frac{\mathcal{U}}{\mathcal{V}} > \frac{1}{\frac{28}{25}\cdot \prod_{5\leq i\leq r+1} (1-\frac{1}{p_i})-1},$$
we get
\begin{align*}
    I(N)&< \frac{3}{2}\cdot \frac{5}{4}\cdot \prod_{5 \leq j\leq (r+1)} \frac{p_j}{p_j-1}\cdot(1+ \frac{\mathcal{V}}{\mathcal{U}})<\frac{3}{2}\cdot \frac{5}{4}\cdot \prod_{5 \leq j\leq (r+1)} \frac{p_j}{p_j-1}\cdot \frac{28}{25}\cdot \prod_{5\leq i\leq r+1} (1-\frac{1}{p_i})=\frac{21}{10}.
\end{align*}
Therefore, for $q_r\geq p_{L}$, $N$ can not be a friend of $20$. Hence, necessarily $q_r<p_{L}$. This completes the proof.

\subsection{Proof of Theorem \ref{newthm}}
Let $P$ be the largest prime divisor of $N$. Since $v_P(N)$ is even, if $v_P(N)\geq 4$ then we are done, therefore, we may assume $v_P(N)=2$. Since $(P^2,\sigma(P^2))=1$, $5\nmid P$ and $\sigma(N)=\frac{21}{10}N$ we have that $5P^2\sigma(P^2)\mid \sigma(N)$. Then
$$5P^2\sigma(P^2)<\sigma(N)=\frac{21}{10}N$$
since $\sigma(P^2)>P^2$ we obtain
$$5P^4<\frac{21}{10}N$$
that is
$$P<(\frac{21}{50})^{\frac{1}{4}}N^{\frac{1}{4}}<N^{\frac{1}{4}}.$$
This completes the proof.

\section{Appendix}
To deduce possible prime divisors of friend of $20$, at first we need to discard the prime numbers for which 
$I(N)<\frac{21}{10}.$
Since the friend $N$ of $20$ is of the form $$N=2\cdot5^{2a}\cdot\prod_{i=1}^{s-2}p_i^{2a_i}$$ 
at first we have to deduce necessary upper bound for $p_1$. The simple algorithm to do so is that, find the least $q_k=p_1$, where $q_k$ is the $k^{th}$ prime, so that
$$I(N)\leq I(2\cdot5^{2a}\cdot\prod_{i=k}^{s-2}q_i^{2a_i})<\frac{3\cdot 5}{2\cdot 4}\cdot \prod_{i=k}^{s-2}\frac{q
_i}{q_i-1}<\frac{21}{10}.$$
For example, suppose that $\omega(N)=4$. Note that if $p_1=19$ then
$$I(2\cdot5^{2a}\cdot\prod_{i=1}^{2}p_i^{2a_i})=I(2\cdot5^{2a}\cdot19^{a_1}\cdot p_2^{2a_2})\leq I(2\cdot5^{2a}\cdot19^{a_1}\cdot 23^{2a_2})<\frac{3\cdot 5\cdot 19\cdot 23}{2\cdot 4\cdot 18\cdot 22}<\frac{21}{10}.$$
Since $19$ is the smallest prime for which $I(N)<\frac{21}{10}$, we take $19$ to be an upper bound for $p_1$. To bound $p_2$, observe that it depends on $p_1$, in fact we have three possible choices for $p_1$. The algorithm to obtain an upper bound for $p_2$ is, fix $p_1$ and choose the least $q_r=p_2$, where $q_r$ is the $r^{th}$ prime, so that 
$$I(N)\leq I(2\cdot5^{2a}\cdot p_1^{2a_1}\cdot \prod_{i=r}^{s-3}q_i^{2a_i})<\frac{3\cdot 5\cdot p_1}{2\cdot 4\cdot (p_1-1)}\cdot \prod_{i=r}^{s-3}\frac{q
_i}{q_i-1}<\frac{21}{10}.$$
For example, suppose that $\omega(N)=4$. Since $p_1\leq 17$, we fix $p_1=17$. Then we choose $p_2=q_r$($r^{th}$ prime) so that
$$I(2\cdot5^{2a}\cdot\prod_{i=1}^{2}p_i^{2a_i})=I(2\cdot5^{2a}\cdot17^{2a_1}\cdot q_{r}^{2a_2})<\frac{3\cdot 5\cdot 17\cdot q_r}{2\cdot 4\cdot 16\cdot (q_{r}-1)}<\frac{21}{10}.$$
Calculating the inequality, we get that $19<q_4$. It follows that we can't choose $p_2\geq 23$, whenever we choose $p_1=17$. Therefore for $p_1=17$ we have $p_2\leq 19$.\\
To obtain an upper bound for $p_3$, we fix $p_1,p_2$ and choose smallest $l^{th}$ prime $q_l=p_3$ so that 
$$I(N)\leq I(2\cdot5^{2a}\cdot p_1^{2a_1}\cdot p_2^{2a_2} \prod_{i=l}^{s-4}q_i^{2a_i})<\frac{3\cdot 5\cdot p_1\cdot p_2}{2\cdot 4\cdot (p_1-1)\cdot (p_2-1)}\cdot \prod_{i=l}^{s-4}\frac{q
_i}{q_i-1}<\frac{21}{10}.$$
Note that, the upper bound for $p_3$ also depends on the choices of $p_1$ and $p_2$. Repeating this method we eventually get an upper bound for every $p_i$.\\\\

The following table contains all necessary $f_{p}^{q}$, for primes $p$ and $q$ that are used in the proof of Theorem \ref{thm 1.1}.
\begin{table}[h]
\begin{center}
\begin{tabular}{|*{5}{c|}}
\hline
$f_{11}^{5}=5$ & $f_{31}^{5}=3$ & $f_{71}^{5}=5$ & $f_{59}^{5}=29$ & $f_{35671}^{5}=29$\\  
\hline
$f_{5}^{11}=5$ & $f_{3221}^{11}=5$ & $f_{7}^{11}=3$ & $f_{19}^{11}=3$ & $f_{8971}^{5}=23$\\
\hline
$f_{7}^{71}=7$ & $f_{883}^{71}=7$ & $f_{7}^{151}=3$ & $f_{3}^{151}=3$ & $f_{7}^{191}=3$\\
\hline
$f_{31}^{191}=3$ & $f_{7}^{211}=7$ & $f_{307189}^{211}=7$ & $f_{7}^{281}=7$ & $f_{29}^{281}=7$\\
\hline
$f_{7}^{331}=3$ & $f_{3}^{331}=3$ & $f_{7}^{401}=3$ & $f_{23029}^{401}=3$ & $f_{7}^{421}=7$\\
\hline
$f_{797310237403261}^{421}=7$ & $f_{7}^{491}=7$ & $f_{617}^{491}=3$ & $f_{7}^{541}=3$ & $f_{3}^{541}=3$\\
\hline
$f_{7}^{571}=3$ & $f_{3}^{571}=3$ & $f_{7}^{631}=7$ & $f_{6032531}^{631}=7$ & $f_{7}^{641}=3$\\
\hline
$f_{58789}^{641}=3$ & $f_{7}^{701}=7$ & $f_{16975792017452101}^{701}=7$ & $f_{7}^{751}=3$ & $f_{3}^{751}=3$\\
\hline
$f_{7}^{911}=7$ & $f_{81750272028928231}^{911}=7$ & $f_{7}^{991}=3$ & $f_{3}^{991}=3$ & $f_{7}^{1031}=3$\\
\hline
$f_{97}^{1031}=3$ & $f_{7}^{1051}=7$ & $f_{29}^{1051}=7$ & $f_{7}^{1061}=3$ & $f_{160969}^{1061}=3$\\
\hline
$f_{5}^{61}=5$ & $f_{7}^{67}=3$ & $f_{3}^{67}=3$ & $f_{5}^{71}=5$ & $f_{2221}^{71}=5$\\
\hline
$f_{7}^{79}=3$ & $f_{3}^{79}=3$ & $f_{5}^{101}=5$ & $f_{491}^{101}=5$ & $f_{7}^{107}=3$\\
\hline
$f_{13}^{107}=3$ & $f_{5}^{31}=5$ & $f_{17351}^{31}=5$ & $f_{911}^{31}=65$ & $f_{11}^{31}=5$\\
\hline
$f_{991}^{31}=99$ & $f_{3}^{31}=3$ & $f_{2011}^{31}=335$ & $f_{11}^{31}=5$ & $f_{79}^{5}=39$\\
\hline
$f_{31}^{5}=3$ & $f_{109}^{5}=27$ & $f_{19}^{5}=9$ & $f_{101}^{5}=25$ & $f_{71}^{5}=5$\\
\hline
$f_{7}^{23}=3$ & $f_{79}^{23}=3$ & $f_{7}^{29}=7$ & $f_{88009573}^{29}=7$ & $f_{7}^{37}=3$\\
\hline
$f_{67}^{37}=3$ & --- & --- & --- & --- \\
\hline
\end{tabular}
\caption{\textcolor{Mycolor1}{List of $f_{p}^q$}}
\end{center}
\end{table}
\section{Concluding Remarks}

 In conclusion, we have established that if $N$ is a friend of $20$, it must have at least six distinct prime factors. This condition significantly narrows the search space for potential friends of $20$, while also highlighting the complexity of identifying such a number. Moreover, we have provided necessary upper bounds for all prime divisors of friends of 20 which can be used to determine which prime numbers can occur in the prime factorization of friend of 20 for some choices of first few smallest prime divisors of friend of 20.

The methods used to prove $\omega(N)\geq 6$ can be further applied to eliminate cases such as $\omega(N)=6,7$ and so on, although this approach may be lengthy but it is expected to be effective.

\section{Data Availablity} 	
The authors confirm that their manuscript has no associated data.

\section{Competing Interests}
The authors confirm that they have no competing interest. 
\section*{Acknowledgement}
The authors are grateful to the anonymous referee for the careful reading of the manuscript and for the valuable suggestions that greatly improved the clarity and presentation of the paper. The authors also thank the referee for Theorem \ref{newthm}. This work was initiated during the visit of the third author to the first author as part of a summer research internship (Ref. No. IITRPR/CEOA/SI/2024/68), held from May 28, 2024, to July 15, 2024, at IIT Ropar. The second author contributed to the work remotely during this period.

 \end{document}